\documentclass[10pt]{article}

\usepackage{babel}
\usepackage[utf8]{inputenc}
\usepackage[shortlabels]{enumitem}
\usepackage[a4paper,bottom=4cm]{geometry}
\usepackage[noend]{algorithmic}
\usepackage{algorithm}
\usepackage{mathtools} 
\usepackage{booktabs}   
\usepackage{amsfonts}
\usepackage{amsthm}
\usepackage{graphicx}
\usepackage{subcaption}
\usepackage{amsfonts} 
\usepackage{placeins} 
\usepackage{amsmath}
\usepackage{amssymb}  
\usepackage{float}
\usepackage{url}
\usepackage{bm}
\usepackage[font=small]{caption}
\usepackage{microtype}
\usepackage[]{fancyhdr} 
\numberwithin{equation}{section}

\DeclareMathOperator{\rank}{rank}
\DeclareMathOperator{\tr}{trace}

\theoremstyle{plain}
\newtheorem{theorem}{Theorem}[section]
\newtheorem{lemma}[theorem]{Lemma}
\newtheorem{corollary}[theorem]{Corollary}
\newtheorem{proposition}[theorem]{Proposition}

\theoremstyle{definition}
\newtheorem{definition}[theorem]{Definition}

\theoremstyle{remark}
\newtheorem{remark}[theorem]{Remark}

\title{Time\,-Varying Semidefinite Programming:\\
Path Following a Burer--Monteiro Factorization}

\author{Antonio Bellon\thanks{Faculty of Electrical Engineering, Czech Technical University in Prague, Karlovo Namesti 13, Prague 121 35, the Czech Republic}
\and Mareike Dressler\thanks{School of Mathematics and Statistics, University of New South Wales, Sydney, NSW 2052, Australia}
\and Vyacheslav Kungurtsev${}^*$
\and Jakub Mareček${}^*$
\and André Uschmajew\thanks{Institute of Mathematics \& Centre for Advanced Analytics and Predictive Sciences, University of Augsburg, 86159 Augsburg, Germany}
}
\date{}

\begin{document}

\maketitle

\begin{abstract}
  We present an online algorithm for time-varying semidefinite programs (TV-SDPs), based on the tracking of the solution trajectory of a low-rank matrix factorization, also known as the Burer--Monteiro factorization, in a path-following procedure. There, a predictor-corrector algorithm solves a sequence of linearized systems. This requires the introduction of a horizontal space constraint to ensure the local injectivity of the low-rank factorization. The method produces a sequence of approximate solutions for the original TV-SDP problem, for which we show that they stay close to the optimal solution path if properly initialized. Numerical experiments for a time-varying max-cut SDP relaxation demonstrate the computational advantages of the proposed method for tracking TV-SDPs in terms of runtime compared to off-the-shelf interior point methods.
\end{abstract}



\section{Introduction}
\label{sec: introduction}

Semidefinite programs (SDPs) constitute an important class of convex constrained optimization problems that is ubiquitous in statistics, signal processing, control systems, and other areas.
In several applications, the data of the problem vary over time, so that this can be modeled as a \emph{time-varying SDP} (TV-SDP).
In this paper we consider TV-SDPs of the form 
\begin{equation}\label{eq: SDP}
    \tag{SDP\textsubscript{$t$}}
    \begin{aligned}
        &\min_{X \in\mathbb{S}^n} &&\langle C_t,X\rangle\\
        &\text{\ \ s.t.}&&\mathcal{A}_t(X)=b_t,
       \\
        &&&X\succeq 0,
    \end{aligned}
\end{equation}
where $t\in[0,T]$ is a time parameter varying on a bounded interval. 
Here $\mathbb{S}^n$ denotes the space of real symmetric $n \times n$ matrices, $\mathcal{A}_t\colon \mathbb{S}^n \to \mathbb R^m$ is a linear operator defined by $\mathcal{A}_t(X)=(\langle A_{1,t},X\rangle,\dots,\langle A_{m,t},X\rangle)$ for some $A_{1,t},\dots,A_{m,t}\in\mathbb{S}^n$, $b_t\in\mathbb{R}^m$, and $C_t\in\mathbb{S}^n$. 
Throughout the paper $\langle \cdot,\cdot \rangle$ denotes the Frobenius inner product and the constraint $X\succeq0$ requires $X\in\mathbb{S}^n$ to be positive semidefinite. 
In this time-varying setting, one looks for a solution curve $t \mapsto X_t$ in $\mathbb S^n$ such that $X = X_t$ is an optimal solution for~\eqref{eq: SDP} at each time point $t \in [0,T]$.

Time-dependent problems leading to TV-SDPs occur in various applications, such as optimal power flow problems in power systems \cite{Lavaei2011}, state estimation problems in quantum systems \cite{Aaronson2019}, modeling of energy economic problems~\cite{Dantzig1977}, job-shop scheduling problems~\cite{AndersonPhD}, as well as problems arising in signal processing, queueing theory~\cite{Nazarathy2009}, or aircraft engineering~\cite{Teren1977}. 
TV-SDPs can be seen as a generalization of continuous linear programming problems, which were first studied by Bellman~\cite{Bellman1953} in relation to so-called bottleneck problems in multistage linear production economic processes. Since then, a large body of literature has been devoted to studying continuous linear programs with and without additional assumptions.
However, the generalization of this idea to other classes of optimization problems has only recently been considered. 
In~\cite{Wang2009}, Wang, Zhang, and Yao~study continuous conic programs, and finally Ahmadi and Khadir~\cite{Ahmadi2021} consider time-varying SDPs. In contrast to our setting, they require the data to vary polynomially with time and also restrict themselves to polynomial solutions. Moreover, the problems studied there involve kernel terms and more complicated constraints, while our work addresses TV-SDPs in a simpler sense of univariate parametric SDPs, following the literature thread of \cite{Goldfarb1999,Hauenstein2022}.

A naive approach to solve the time-varying problem~\eqref{eq: SDP} is to consider, at a sequence of times $\{t_k\}_{k\in\{1,\dots,K\}}\subseteq[0,T]$, the instances of the problem (SDP$_{t_k}$) for $k\in\{1,\dots,K\}$ and solve them one after another. 
The best solvers for SDPs are interior point methods \cite{jarre1993interior,Tuncel2000,Andersen2003, Freund2006,Jiang2020}, which can solve them in a time that is polynomial in the input size. However, these solvers do not scale particularly well, and thus this brute-force approach may fail in applications where the volume and velocity of the data are large. 
Furthermore, such a straightforward method would not make use of the local information collected by solving the previous instances of the problem. 
Even if one considers warm starts~\cite{gondzio2002reoptimization,Gondzio2008,Engau2010,Colombo2011,Skajaa2013}, the reduction in run time is likely to be marginal. 
For instance,~\cite[sections 5.5 and 5.6]{Skajaa2013} reports a 30--60\% reduction of the runtime on a collection of time-varying instances of their own choice. 

Instead, in this work, we would like to utilize the idea of so-called path-following predictor-corrector algorithms as developed in~\cite{Guddat1990,Allgower1990}. 
In classical predictor-corrector methods, a predictor step for approximating the directional derivative of the solution with respect to a small change in the time parameter is applied, together with a correction step that moves from the current approximate solution closer to the next solution at the new time point. 
The latter is based on a Newton step for solving the first-order optimality KKT conditions.

A limiting factor in solving both stationary and time-dependent SDPs is computational complexity when $n$ is large. 
A common solution to this obstacle is the Burer--Monteiro approach, as presented in the seminal work~\cite{Burer2003,Burer2005}. 
In this approach, a low-rank factorization $X = Y Y^T$ of the solution is assumed with $Y \in \mathbb R^{n \times r}$ and $r$ potentially much smaller than $n$. 
In the optimization literature, the Burer--Monteiro method has been very well studied as a nonconvex optimization problem, e.g., in terms of algorithms~{\cite{Journee2010}}, quality of the optimal value~\cite{Barvinok1995,Pataki1998}, and (global) recovery guarantees~\cite{Boumal2016, Boumal2020,Cifuentes2021}.

In a time-varying setting, the Burer--Monteiro factorization leads to 
\begin{equation}
\label{eq: BM}
\tag{BM\textsubscript{$t$}}
    \begin{aligned}
        &\min_{Y\in\mathbb{R}^{n\times r}} &&\langle C_t,YY^T\rangle
        \\
        &\text{\ \ \ s.t.}&&\mathcal{A}_t(YY^T)=b_t,
    \end{aligned}
\end{equation}
which for every fixed $t$ is a quadratically constrained quadratic problem. 
A solution then is a curve $t \mapsto Y_t$ in $\mathbb R^{n \times r}$, which, depending on $r$, is a space of much smaller dimension than $\mathbb S^n$. 
However, this comes at the price that the problem~\eqref{eq: BM} is now nonconvex. 
Moreover, theoretically it may happen that local optimization methods converge to a critical point that is not globally optimal~\cite{Waldspurger2020}, although in practice the method usually shows very good performance~\cite{Burer2003,Journee2010,Rosen2020}.

The aim of this work is to combine the Burer--Monteiro factorization with path-following predictor-corrector methods and to develop a practical algorithm for approximating the solution of~\eqref{eq: BM}, and consequently of~\eqref{eq: SDP}, over time. 
As we explain in section~\ref{sec: quotient_geometry}, to apply such methods, we need to address the issue that the solutions of~\eqref{eq: BM} are never isolated, due to the nonuniqueness of the Burer--Monteiro factorization caused by orthogonal invariance. 
In this paper, we apply a well-known technique to handle this problem by restricting the solutions to a so-called \textit{horizontal space} at every time step. 
From a geometric perspective, such an approach exploits the fact that equivalent factorizations can be identified as the same element in the corresponding quotient manifold with respect to the orthogonal group action~\cite{Massart2020}. 

The paper is structured as follows. In section~\ref{sec: main_assumptions} we review important foundations from~the SDP literature and state the main assumptions we make on the TV-SDP problem~\eqref{eq: SDP}.  
Section~\ref{sec: quotient_geometry} presents the underlying quotient geometry of positive semidefinite rank-$r$ matrices from a linear algebra perspective, focusing in particular on the notion of horizontal space and the domain of injectivity of the map $Y\mapsto YY^T$. We then describe in section~\ref{sec: predictor_corrector} our path-following predictor-corrector algorithm, which is based on iteratively solving the linearized KKT system for~\eqref{eq: BM} over time. 
A main result is the rigorous error analysis for this algorithm presented in subsection~\ref{subsec: err_analysis}. 
In section~\ref{sec: num_exp}, we showcase numerical results that test our method on a time-varying variant of the well-known Goemans--Williamson SDP relaxation for the Max-Cut problem in combinatorial optimization and graph theory. 
We conclude in section~\ref{sec: conclusion} with a brief discussion of our results.

\section{Preliminaries and key assumptions}
\label{sec: main_assumptions}

Naturally, the rigorous formulation of path-following algorithms requires regularity assumptions on the solution curve. 
In our context, this will require both assumptions on the original TV-SDP problem~\eqref{eq: SDP} as well as on its reformulation~\eqref{eq: BM}.
In particular for the latter, the correct choice of the dimension $r$ is crucial.
In what follows, we present and discuss these assumptions in~detail. 

First, we briefly review some standard notions and properties for primal-dual SDP pairs; see~\cite{Alizadeh1997,handbookSDP}. Consider the \textit{conic dual problem} of~\eqref{eq: SDP}:
\begin{equation}
\label{eq: DSDP}
\tag{D-SDP\textsubscript{$t$}}
    \begin{aligned}
        &\max_{w\in\mathbb{R}^m} &&\langle b_t, w\rangle\\
        &\text{\ \ s.t.}&&Z(w)\coloneqq C_t-\mathcal A^*_t(w)\succeq0
    \end{aligned}
\end{equation}
where $\mathcal A_t^*\colon w \mapsto \sum_{i=1}^m w_iA_{i,t}$ is the linear operator adjoint to $\mathcal A_t$.
For convenience, we often drop the explicit dependence on $w$ and refer to a solution of~\eqref{eq: DSDP} simply as $Z$.
While reviewing the basic properties of SDPs, we assume the time parameter to be fixed and hence omit the subindex~$t$.

The KKT conditions for the pair of primal-dual convex problems~\eqref{eq: SDP}-\eqref{eq: DSDP} read
\begin{equation}
    \label{eq: KKT_SDP}
    \begin{aligned}
        \mathcal A(X)&=b, 
        &&X\succeq0,
        \\
        Z +\mathcal A^*(w)&=C,
        &&Z\succeq0,\\
        XZ&=0.
    \end{aligned}
\end{equation}
These are sufficient conditions for the optimality of the pair $(X,Z)$.
\begin{definition}[strict feasibility]
\label{def: strict_feas}
We say that \textit{strict feasibility} holds for an instance of primal SDP if there exists a positive definite matrix $X\succ0$ that satisfies $\mathcal{A}(X) = b$. Similarly, strict feasibility holds for the dual if there exist a vector $w\in\mathbb{R}^m$ satisfying $Z(w)\succ 0$.
 
\end{definition}
It is well-known that under strict feasibility the KKT conditions are also necessary for optimality.
Note that, in general, a pair $(X,Z)$ of optimal solutions satisfies the inclusions $\operatorname{im}X\subseteq\operatorname{ker}Z$ and $\operatorname{im}Z\subseteq\operatorname{ker}X$, where $``\operatorname{im}"$ and $``\operatorname{ker}"$ denote the image and kernel, respectively. 

\begin{definition}[strict complementarity]
\label{def: strict_compl}
A primal-dual optimal point $(X,Z)$ is said to be \textit{strictly complementary} if $\operatorname{im}X=\operatorname{ker}Z$ (or, equivalently, $\operatorname{im}Z=\operatorname{ker}X$). 
A primal-dual pair of an instance of SDP satisfies \textit{strict complementarity} if there exists a strictly complementary primal-dual optimal point $(X,Z)$. 
\end{definition}

\begin{definition}[nondegeneracy]
\label{def: non_deg}
A primal feasible point $X$ is \textit{primal nondegenerate} if
\begin{equation}
\label{eq: primal_non_deg}
    \ker\mathcal{A}+\mathcal{T}_{X}=\mathbb{S}^{n},
\end{equation}
with $\mathcal{T}_{X}$ being the tangent space to the manifold $\mathcal M_r$ of fixed rank-$r$ symmetric matrices at $X$, where $r = \rank X$.
Let $X = YY^T$ be a rank-revealing decomposition, then 
\[
    \mathcal{T}_{X} = \{ Y V^T + V Y^T \colon V \in \mathbb R^{n \times r} \}.
\]
A dual feasible point $Z$ is \textit{dual nondegenerate} if
\begin{equation*}
    \operatorname{im} \mathcal A^*+\mathcal{T}_{Z}=\mathbb{S}^{n},
\end{equation*}
where $\mathcal{T}_{Z}$ is the tangent space at $Z$ to the manifold of fixed rank-$s$ symmetric matrices with $s = \rank Z$.
\end{definition}

Primal-dual strict feasibility implies the existence of both a primal and a dual optimal solution with a zero duality gap.
In addition, primal (dual) nondegeneracy implies dual (primal) uniqueness of the solutions. 
Under strict complementarity, the converse is also true, that is, primal (dual) uniqueness of the primal dual optimal solutions implies dual (primal) nondegeneracy of these solutions.
Moreover, primal-dual nondegeneracy and strict complementarity hold generically. We refer to~\cite{Alizadeh1997} for details. 

With regard to the time-varying case, these facts can be generalized as follows.

\begin{theorem}{\normalfont(Bellon et al., \cite[Theorem 2.19]{Bellon2021})}
\label{thm:single_valued_differentiable}
Let {\upshape (P\textsubscript{$t$},D\textsubscript{$t$})} be a primal-dual pair of TV-SDPs parametrized over a time interval $[0,T]$ such that primal-dual strict feasibility holds for any $t\in [0,T]$ and assume that the data $\mathcal{A}_t,b_t,C_t$ are continuously differentiable functions of $t$.
Let $t^* \in [0,T]$ be a fixed value of the time parameter and suppose that $(X^*, Z^*)$ is a nondegenerate optimal and strictly complementary point for {\upshape (P\textsubscript{$t^*$},D\textsubscript{$t^*$})}.
Then there exists $\varepsilon > 0$ and a continuously differentiable unique mapping $t\mapsto(X_t, Z_t)$ defined on $(t^*-\varepsilon,t^*+\varepsilon)$ such that $(X_t, Z_t)$ is a unique and strictly complementary primal-dual optimal point to {\upshape (P\textsubscript{$t$},D\textsubscript{$t$})} for all $t\in(t^*-\varepsilon,t^*+\varepsilon)$.
In particular, the ranks of $X_t$ and $Z_t$ are constant for all $t\in(t^*-\varepsilon,t^*+\varepsilon)$.
\end{theorem}
The last statement of the theorem directly follows from the fact that a change in the rank of either $X_t$ or $Z_t$ implies a loss of strict complementarity because of the lower-semicontinuity of the rank.
Based on these facts, for the initial problem~\eqref{eq: SDP} we make the following assumptions.
\begin{enumerate}[label={(A\arabic*)}]
    \item\label{ass: strict_feas}
    \eqref{eq: SDP} and \eqref{eq: DSDP} are strictly feasible for any $t\in[0,T]$.
    \item\label{ass: licq}
    The linear operator $\mathcal{A}_t$ is surjective in any $t\in[0,T]$.
    \item\label{ass: non_deg}
    \eqref{eq: SDP} has a primal nondegenerate solution $X_t$ and\\~\eqref{eq: DSDP} has a dual nondegenerate solution $Z_t$ at any $t\in[0,T]$.
    \item\label{ass: strict_compl}
    The solution pair $(X_t,Z_t)$ is strictly complementary for any $t\in[0,T]$.
    \item\label{ass: cont_diff}
    Data $\mathcal{A}_t,b_t,C_t$ are continuously differentiable functions of $t$.
\end{enumerate}
\setlist[enumerate,1]{start=1}
Assumptions~\ref{ass: strict_feas} and~\ref{ass: licq} are standard for SDPs and in linearly constrained optimization in general, while assumptions \ref{ass: non_deg}--\ref{ass: strict_compl} rule out many ``pathological'' cases \cite{Bellon2021}. 
In particular, assumption~\ref{ass: non_deg} implies that the solution pair $(X_t,Z_t)$ is unique. By Theorem~\ref{thm:single_valued_differentiable}, assumptions \ref{ass: non_deg},~\ref{ass: strict_compl}, and~\ref{ass: cont_diff} have the following consequences:
\begin{enumerate}[label={(C\arabic*)}] 
    \item\label{cons: unique}
    \eqref{eq: SDP} has a unique and smooth solution curve $X_t$, $t \in [0,T]$.
    \item\label{cons: rank}
    The curve $t\mapsto X_t$ is of constant rank $r^*$.
\end{enumerate} 
\mbox{}
\newline
For setting up the factorized version~\eqref{eq: BM} of~\eqref{eq: SDP}, it is necessary to choose the dimension $r$ of the factor matrix $Y$ in~\eqref{eq: BM}, ideally equal to $r^*$ of~\ref{cons: rank}.  
In what follows, we assume that we know the constant rank $r^*$.
Given access to an initial solution $X_0$ at time $t=0$, it is possible to compute~$r^*$, so this assumption is without further loss of generality.  

It is worth noting that the rank cannot be arbitrary.
Based on a known result of Barvinok and Pataki~\cite{Barvinok2001}, for any SDP defined by $m$ linearly independent constraints, there always exists a solution of rank $r$ such that $\frac{1}{2}r(r+1)\le m$.
Since we assume that $X_t$ is the unique solution to~\eqref{eq: SDP} with constant rank $r^*$ we conclude that 
\[
    \frac{1}{2}r^*(r^*+1)\le m.
\]
We point out that recently the Barvinok--Pataki bound has been slightly improved~\cite{Im2021}.
 
\section{Quotient geometry of positive semidefinite rank-$\bm{r}$ matrices} 
\label{sec: quotient_geometry}

We now investigate the factorized formulation~\eqref{eq: BM} in more detail. As already mentioned, in contrast to the original problem~\eqref{eq: SDP}, this is a nonlinear problem (specifically, a quadratically constrained quadratic problem) which is nonconvex.
Moreover, the property of uniqueness of a solution, which is guaranteed by~\ref{cons: unique} for the original problem~\eqref{eq: SDP}, is lost in~\eqref{eq: BM}, because its representation via the map
\[
    \phi : \mathbb R^{n \times r} \to \mathbb S^n, \quad \phi(Y) = Y Y^T
\]
is not unique.
In fact, this map is invariant under the orthogonal group action
\[
    \mathcal{O}_r\times \mathbb{R}^{n\times r}\to \mathbb{R}^{n\times r}, \quad (Q,Y)\mapsto YQ,
\]
on $\mathbb R^{n \times r}$, where
\[
    \mathcal{O}_r:=\{Q\in\mathbb{R}^{r\times r}\ :\ QQ^T=I_r\}, 
\]
with $I_r$ denoting the $r\times r$  identity matrix, is the orthogonal group. Hence both the objective function $Y\mapsto\langle C_t,YY^T\rangle$ and the constraints $\mathcal{A}_t(YY^T)=b_t$ in~\eqref{eq: BM} are invariant under the same action.
As a consequence, the solutions of~\eqref{eq: BM} are never isolated \cite{Journee2010}.
This poses a technical obstacle to the use of path-following algorithms, as the path needs to be, at least locally, uniquely defined.

On the other hand, by assuming that the correct rank $r = r^*$ of a unique solution $X_t$ for \eqref{eq: SDP} has been chosen for the factorization, any solution $Y_t$ for \eqref{eq: BM} must satisfy $Y^{}_t Y_t^T = X_t$. From this it follows that any solution is of the form $Y_t Q$ with $Q \in \mathcal O_r$; see, e.g.,~\cite[Lemma~2.1]{Burer2005}. In other words, the action of the orthogonal group is indeed the only source of nonuniqueness.
This corresponds to the well-known fact that the set of positive definite fixed rank-$r$ symmetric matrices, which we denote by $\mathcal M_r^+$, is a smooth manifold  that can be identified with the quotient manifold $\mathbb R_*^{n \times r} / \mathcal O_r$, where $\mathbb R_*^{n \times r}$ is the open set of $n \times r$ matrices with full column rank.

In the following, we describe how the nonuniqueness can be removed by introducing a so-called horizontal space, which is a standard concept in optimization on quotient manifolds, see,~e.g.,~\cite[section~3.5.8]{Absil2008}. For positive semidefinite fixed-rank matrices, this has been worked out in detail in~\cite{Massart2020}.
Additional material, including the complex Hermitian case, can be found in~\cite{Balan2021}.
However, in order to arrive at practical formulas that are useful for our path-following algorithm later on, we will not further refer to the concept of a quotient manifold but directly focus on the injectivity of the map $\phi$ on suitable linear subspaces of $\mathbb R^{n \times r}$, which we describe in the following section. Such a simplification takes into account that we are dealing with a quotient manifold $\mathbb R_*^{n \times r} / \mathcal O_r$ with $\mathbb R_*^{n \times r}$ being just an open subset of $\mathbb R^{n \times r}$. Then the horizontal space at a point $Y$ should be a subspace of the tangent space of $\mathbb R_*^{n \times r}$ at $Y$, which, however, is just $\mathbb R^{n \times r}$.

\subsection{Horizontal space and unique factorizations}
\label{subsec: hor_space}

Given $Y \in \mathbb R^{n \times r}_*$, we denote the corresponding orbit under the orthogonal group as
\[
    Y\mathcal{O}_r:=\{Y Q \, :\, Q \in \mathcal O_r\} \subseteq\mathbb{R}^{n\times r}_*.
\]
The orbit $Y \mathcal{O}_r$ is an embedded submanifold of $\mathbb{R}^{n\times r}_*$ of dimension $\tfrac{1}{2}r(r-1)$ with two connected components, according to $\operatorname{det}Q=\pm1$. 
Its tangent space at $Y$, which we denote by $\mathcal T_{Y}$, is easily derived by noting that the tangent space to the orthogonal group $\mathcal O_r$ at the identity matrix equals the space of real skew-symmetric matrices $\mathbb{S}^r_{skew}$ (see, e.g.,~\cite[Example~3.5.3]{Absil2008}).
Therefore,
\begin{equation*}
\label{eq: orbit tangent space}
    \mathcal T_{Y} =\{Y S\, :\, S\in \mathbb{S}^r_{skew}\}.
\end{equation*}
Since the map $\phi(Y) = Y Y^T$ is constant on $Y \mathcal{O}_r$, its derivative
\[
     Y \mapsto \phi'(Y)[H] = Y H^T + H Y^T
\]
vanishes on $\mathcal T_{Y}$, that is $\mathcal{T}_Y\subseteq\operatorname{ker} \phi'(Y)$.

The \emph{horizontal space} at $Y$, denoted by $\mathcal H_{Y}$, is the orthogonal complement of $\mathcal T_{Y}$ with respect to the Frobenius inner product. One verifies that
\begin{equation*}
\label{eq:orth_hor_space}
      \mathcal H_{Y}^{} := \mathcal T_{Y}^\perp =\{ H \in\mathbb{R}^{n\times r}\, :\, Y^T H = H^T Y \},
\end{equation*}
since $0 = \langle H, Y S\rangle = \langle Y^T H, S\rangle$ holds for all skew-symmetric $S$ if and only if $Y^T H$ is symmetric. We point out that sometimes any subspace complementary to $\mathcal T_{Y}$ is called a horizontal space, but we will stick to the above choice, as it is the most common and has certain theoretical and practical advantages.
In particular, since $Y\in\mathcal H_Y$, the affine space $Y+\mathcal H_Y$ equals $\mathcal H_Y$, so it is just a linear space. 

The purpose of the horizontal space is to provide a unique way of representing a neighborhood of $X = Y Y^T$ in $\mathcal M_r^+$ through $\phi(Y + H) = (Y + H)(Y + H)^T$ with $H \in \mathcal H_{Y}$. 
Clearly,
\[
    \dim \mathcal H_{Y} = nr - \dim \mathcal O_r = nr - \frac{1}{2}r(r-1) = \dim \mathcal M_r^+.
\]
Moreover, the following holds.

\begin{proposition}
\label{prop: injectivity Dphi}
The restriction of $\phi'(Y)$ to $\mathcal H_{Y}$ is injective. In particular, it holds that
\[
    \| Y H^T + H Y^T \|_F \ge \sqrt{2}\sigma_r(Y) \| H \|_F \quad \text{for all $H \in \mathcal H_{Y}$,}
\]
where $\sigma_r(Y) > 0$ is the smallest singular value of $Y$.
This lower bound is sharp if $r<n$. For $r=n$ one has the sharp estimate
\[
    \| Y H^T + H Y^T \|_F \ge 2\sigma_r(Y) \| H \|_F \quad \text{for all $H \in \mathcal H_{Y}$.}
\]
As a consequence, in either case, $\operatorname{ker}\phi'(Y)=\mathcal{T}_Y$.
\end{proposition}

\begin{proof}
For $Z \in \mathbb S^n$ we have $\tr( (Y H^T + H Y^T) Z ) = 2 \tr(Z Y H^T)$ by standard properties of the trace. Taking $Z = Y H^T + H Y^T$ yields
\begin{align*}
    \| Y H^T + H Y^T \|_F^2 &= 2 \tr( Y H^T Y H^T + H Y^T Y H^T ) = 2 \| Y^T H \|_F^2 + 2 \| Y H^T \|_F^2.
\end{align*}
To derive the second equality we used $Y^T H = H^T Y$ for $H \in \mathcal H_Y$. 
Clearly, $\| Y H^T \|_F^2 \ge \sigma_r(Y)^2 \| H \|_F^2$ and if $r = n$ we also have that $\| Y^T H \|_F^2 \ge \sigma_r(Y)^2 \| H \|_F^2$.
This proves the asserted lower bounds.
To show that they are sharp, let $(u_r,v_r)$ be a (normalized) singular vector tuple such that $Y v_r = \sigma_r(Y) u_r$.
If $r < n$, then for any $u$ such that $u^T Y = 0$ one verifies that the matrix $H = u v_r^T$ is in $\mathcal H_Y$ and achieves equality. When $r=n$, $H = u_r^{} v_r^T$ achieves it. 
\end{proof}

Since $\phi$ maps $\mathbb R_*^{n\times r}$ to $\mathcal M_r^+$, which is of the same dimension as $\mathcal H_{Y}$, the above proposition implies that $\phi'(Y)$ is a bijection between $\mathcal H_{Y}$ and $\mathcal T_{\phi(Y)} \mathcal M_r^+$.
This already shows that the restriction of $\phi$ to the linear space $Y + \mathcal H_{Y} = \mathcal H_Y$ is a local diffeomorphism between a neighborhood of $Y$ in $\mathcal H_{Y}$ and a neighborhood of $\phi(Y)$ in $\mathcal M_r^+$.
The subsequent more quantitative statement matches Theorem 6.3 in~\cite{Massart2020} on the injectivity radius of the quotient manifold $\mathbb R_*^{n \times r} / \mathcal O_r$.
For convenience we will provide a self-contained proof that is more algebraic and does not require the concept of quotient manifolds.

\begin{proposition}
\label{prop: injectivity of phi}
Let $\mathcal B_Y := \{ H \in \mathcal H_{Y} \colon \| H \|_F < \sigma_r(Y) \}$. 
Then the restriction of $\phi$ to $Y + \mathcal B_Y$ is injective and maps diffeomorphically to a (relatively) open neighborhood of $Y$ in $\mathcal M_r^+$.
\end{proposition}

It is interesting to note that $\mathcal B_Y$ is the largest possible ball in $\mathcal H_Y$ on which the result can hold, since the rank-one matrices $\sigma_i u_i v_i^T$ comprised of singular pairs of $Y$ all belong to $\mathcal H_{Y}$ and $Y - \sigma_r u_r v_r^T$ is rank-deficient.
Another important observation is that $\sigma_r(Y)$ does not depend on the particular choice of $Y$ within the orbit $Y \mathcal O_r$.

\begin{proof}
Consider $H_1, H_2 \in \mathcal B_Y$. Let $Y = U \Sigma V^T$ be a singular value decomposition of $Y$ with $U \in \mathbb R^{n \times r}$ and $V \in \mathbb R^{r \times r}$ having orthonormal columns. We assume $r < n$. Then by $U_\perp \in \mathbb R^{n \times (n-r)}$ we denote a matrix with orthonormal columns and $U^T U_\perp = 0$. In the case $r = n$, the terms involving $U_\perp$ in the following calculation are simply not present.  We write
\[
    H_1 = U A_1 V^T + U_\perp B_1 V^T, \quad H_2 = U A_2 V^T + U_\perp B_2 V^T.
\]
Since $H_1, H_2 \in \mathcal H_{Y}$, we have
\[
    \Sigma A_1^{} = A_1^T \Sigma, \quad \Sigma A_2^{} = A_2^T \Sigma.
\]
Then a direct calculation yields
\begin{equation*}
\label{eq: decomposition}
    \begin{aligned}
        (Y + H_1)(Y+H_1)^T - Y Y^T 
        &=U [\Sigma A_1^T + A_1 \Sigma + A_1 A_1^T] U^T 
        \\
        &+U [\Sigma + A_1] B_1^T U_\perp^T + U_\perp B_1 [\Sigma + A_1^T]U^T + U_\perp B_1 B_1^T U_\perp^T,
    \end{aligned}
\end{equation*}
and analogously for $(Y + H_2)(Y+H_2)^T - Y Y^T$.
Since the four terms in the above sum are mutually orthogonal in the Frobenius inner product, the equality $(Y + H_1)(Y+H_1)^T = (Y + H_2)(Y+H_2)^T$ particularly implies
\[
    \Sigma A_1^T + A_1 \Sigma + A_1 A_1^T = \Sigma A_2^T + A_2 \Sigma + A_2 A_2^T,
\]
as well as
\begin{equation}
\label{eq: second term}
    (\Sigma + A_1) B_1^T = (\Sigma + A_2) B_2^T.
\end{equation}
The first of these equations can be written as
\[
    \Sigma (A_1 - A_2)^T + (A_1 - A_2) \Sigma = A_2 (A_2 - A_1)^T - (A_1 - A_2)A_1^T. 
\]
By Proposition~\ref{prop: injectivity Dphi} (with $n=r$, $Y = \Sigma$ and $H = A_1 - A_2$),
\[
    \| \Sigma (A_1 - A_2)^T + (A_1 - A_2) \Sigma \|_F \ge 2 \sigma_r(Y) \| A_1 - A_2 \|_F,
\]
whereas
\begin{equation*}
        \| A_2 (A_2 - A_1)^T - (A_1 - A_2)A_1^T \|_F 
        \le (\| H_2 \|_F + \| H_1 \|_F) \| A_1 - A_2 \|_F.
\end{equation*}
Since $\| H_2 \|_F + \| H_1 \|_F < 2 \sigma_r(Y)$, this  shows that we must have $A_1 = A_2$, which then by~\eqref{eq: second term} also implies $B_1 = B_2$, since $\Sigma + A_1$ is invertible.

Hence, we have proven that $\phi$ is an injective map from $Y + \mathcal B_Y$ to $\mathcal M_r^+$.
To validate that it is a diffeomorphism onto its image we show that it is locally a diffeomorphism, for which again it suffices to confirm that $\phi'(Y + H)$ is injective on $\mathcal H_Y$ for every $H \in \mathcal B_Y$ (since $\mathcal H_Y$ and $\mathcal M_r^+$ have the same dimension).
It follows from Proposition~\ref{prop: injectivity Dphi} (with $Y$ replaced by $Y+H$, which has full column rank) that the null space of $\phi'(Y+H)$ equals $\mathcal T_{Y + H}$.
We claim that $\mathcal T_{Y + H} \cap \mathcal H_{Y} = \{0\}$, which proves the injectivity of $\phi'(Y+H)$ on $\mathcal H_Y$.
Indeed, let $K$ be an element in the intersection, i.e., $K = (Y+H)S$ for some skew-symmetric $S$ and $Y^T K - K^T Y = 0$.
Inserting the first relation into the second, and using $Y^T H = H^T Y$, yields the homogenuous Lyapunov equation
\begin{equation}
\label{eq: Lyapunov eq}
    (Y^T Y + Y^T H) S + S (Y^T Y + Y^T H) = 0.
\end{equation}
The symmetric matrix 
\[
    Y^T Y + Y^T H = \frac{1}{2} (Y + H)^T(Y + H) + \frac{1}{2}(Y^T Y - H^T H)
\]
in~\eqref{eq: Lyapunov eq} is positive definite, since $\lambda_1(H^TH)\le\| H^T H \|_F < \sigma_r(Y)^2 = \lambda_r(Y^T Y)$ (here $\lambda_i$ denotes the $i$-th eigenvalue of the corresponding matrix).
But in this case~\eqref{eq: Lyapunov eq} implies $S = 0$, that is, $K=0$.
\end{proof}

Finally, it is also possible to provide a lower bound on the radius of the largest ball around $X = YY^T$ such that its intersection with $\mathcal M_r^+$ is in the image $\phi(Y + \mathcal B_Y)$ (so that an inverse map $\phi^{-1}$ is defined). 

\begin{proposition}
\label{prop: inverse_inject_radius}
Any $\tilde X \in \mathcal M_r^+$ satisfying $\| \tilde X - X \|_F < \frac{2\lambda_r(X)}{\sqrt{r+4} + \sqrt{r}}$ is in the image $\phi(Y + \mathcal B_Y)$, that is, there exists a unique $H \in \mathcal B_Y$ such that $\tilde X = (Y + H)(Y+H)^T$. 
\end{proposition}

Observe that one could take 
\begin{equation}
\label{eq: clean_inverse_inject_radius}
    \| \tilde X - X \|_F \le \frac{\lambda_r(X)}{\sqrt{r+4}}.
\end{equation}
as a slightly cleaner sufficient condition in the proposition. 

\begin{proof}
Let $\tilde X = \tilde Z \tilde Z^T$ with $\tilde Z\in\mathbb{R}^{n\times r}$ and assume a polar decomposition of $Y^T \tilde Z = P \tilde Q^T$, where $P, \tilde Q \in \mathbb R^{r \times r}$, $P$ is positive semidefinite, and $\tilde Q$ is orthogonal. Let $Z = \tilde Z \tilde Q$. 
Then
\begin{equation}
\label{eq: H_construction} 
    H = Z - Y
\end{equation}
satisfies $(Y+H)(Y+H)^T = \tilde X$, and since $Y^T H = P - Y^T Y$ is symmetric, we have $H \in \mathcal H_Y$. We need to show $H \in \mathcal B_Y$, that is, $\| H \|_F < \sigma_r(Y)$.
Proposition~\ref{prop: injectivity of phi} then implies that $H$ is unique in $\mathcal B_Y$.
Let $Y Y^\dagger$ be the orthogonal projector onto the column span of $Y$ and $Z_1 = Y Y^\dagger Z$. 
With that, we have the decomposition
\begin{equation}
\label{eq: error decomposition}
    \| H \|_F^2 = \| Y Y^\dagger H \|_F^2 + \|(I - Y Y^\dagger) H \|_F^2 = \| Z_1 - Y \|_F^2 + \|(I - Y Y^\dagger) Z \|_F^2.
\end{equation}
We estimate both terms separately.
Since $Y^T  Z_1 = Y^T Z = P$ is symmetric and positive semidefinite, the first term satisfies
\begin{align}
\label{eq: norm first part} 
    \| Z_1 - Y \|_F^2 
    &=  \| Z_1 \|_F^2 - 2 \tr(Y^T Z_1) + \| Y \|_F^2 \notag
    \\ 
    &= \| (Z_1 Z_1^T)^{1/2} \|_F^2 - 2 \sum_{i=1}^r \sigma_i(Y^T Z_1) + \| (Y Y^T)^{1/2} \|_F^2.
\end{align} 
A simple consideration using a singular value decomposition of $Y$ and $Z_1$ reveals that
\[
    (YY^T)^{1/2}(Z_1Z_1^T)^{1/2} = \tilde U Y^T Z_1 \tilde V^T
\]
for some $\tilde U$ and $\tilde V$ with orthonormal columns. Consequently, by von Neumann's trace inequality (see, e.g.,~\cite[Theorem~7.4.1.1]{Horn2013}), we have
\[
    \tr((YY^T)^{1/2}(Z_1Z_1^T)^{1/2}) \le \sum_{i=1}^r \sigma_i(Y^T Z_1).
\]
Inserting this into~\eqref{eq: norm first part} yields
\[
    \| Z_1 - Y \|_F^2 \le \| (Z_1 Z_1^T)^{1/2} - (Y Y^T)^{1/2} \|_F^2.
\]
We remark that we could have concluded this inequality from~\cite[Theorem~2.7]{Balan2021} where it is also stated. It actually holds for any $Z_1$ for which $Y^T Z_1$ is symmetric and positive semidefinite using the same argument (in particular for $Z_1$ replaced with the initial $Z$). Let now $Y = U \Sigma V^T$ be a singular value decomposition of $Y$ with $\sigma_r(Y)$ the smallest positive singular value. Then $Z_1 Z_1^T = U S^2 U^T$ for some positive semidefinite $S^2 \in \mathbb{R}^{r \times r}$ and it follows from well-known results, (cf.~\cite{Schmitt1992}), that\footnote{For completeness we provide the proof.
The matrix $S - \Sigma$ is the unique solution to the matrix equation $\mathcal{L}(M) = SM + M \Sigma = S^2 - \Sigma^2$. 
Indeed, the linear operator $\mathcal{L}$ on $\mathbb R^{r \times r}$ is symmetric in the Frobenius inner product and has positive eigenvalues $\lambda_{i,j} =  \lambda_i(S) + \Sigma_{jj} \ge \sigma_r(Y)$ (the eigenvectors are rank-one matrices $w_i e_j^T$ with $w_i$ the eigenvectors of $S$). 
Hence $\| S^2 - \Sigma^2 \|_F = \|\mathcal{L}(S - \Sigma) \|_F \ge \sigma_r(Y) \| S - \Sigma \|_F$.}
\[
    \| (Z_1 Z_1^T)^{1/2} - (Y Y^T)^{1/2} \|_F^2 = \| S - \Sigma \|_F^2 \le \frac{1}{\sigma_r(Y)^2} \| S^2 - \Sigma^2 \|_F^2 = \frac{1}{\sigma_r(Y)^2} \| Z_1 Z_1^T - Y Y^T \|_F^2.
\]
Noting that $Z_1 Z_1^T = (Y Y^\dagger) \tilde X (Y Y^\dagger)$ and $YY^T = (Y Y^\dagger) X (Y Y^\dagger)$ we conclude the first part with
\begin{equation}
\label{eq: first part}
    \| Z_1 - Y \|_F^2 \le \frac{1}{\sigma_r(Y)^2} \| (Y Y^\dagger)(\tilde X - X)(Y Y^\dagger) \|_F^2 \le \frac{1}{\sigma_r(Y)^2} \| \tilde X - X \|_F^2.
\end{equation}
The second term in~\eqref{eq: error decomposition} can be estimated as follows:
\begin{align}
    \|(I - Y Y^\dagger) Z \|_F^2  
    &= \tr((I - Y Y^\dagger) \tilde X (I - Y Y^\dagger)) \notag 
    \\
    &\le \sqrt{r} \| (I - Y Y^\dagger)\tilde X (I - Y Y^\dagger) \|_F \notag 
    \\
    &= \sqrt{r} \| (I - Y Y^\dagger)(\tilde X - X) (I - Y Y^\dagger) \|_F \le \sqrt{r} \| \tilde X - X \|_F, \label{eq: second part}
\end{align}
where we used the Cauchy Schwarz inequality and the fact that $(I - Y Y^\dagger) \tilde X (I - Y Y^\dagger)$ has rank at most $r$.

As a result, combining~\eqref{eq: error decomposition} with~\eqref{eq: first part} and~\eqref{eq: second part}, we obtain
\begin{equation}
\label{eq: inverse estimate}
    \| H \|_F^2 \le \frac{1}{\sigma_r(Y)^2} \| \tilde X - X \|_F^2 + \sqrt{r}\| \tilde X - X \|_F.
\end{equation}
The right side is strictly smaller  than $\sigma_r(Y)^2$ when  
\[
    \| \tilde X - X \|_F < -\frac{\sigma_r(Y)^2 \sqrt{r}}{2} + \sqrt{\frac{\sigma_r(Y)^4 r}{4} + \sigma_r(Y)^4} = \frac{\sigma_r(Y)^2}{2}(\sqrt{r+4} - \sqrt{r}) = \frac{2 \lambda_r(X)}{\sqrt{r+4} + \sqrt{r}},
\]
which proves the assertion. 
\end{proof}

\begin{remark} 
\label{rem: polar_decomposition}
From definition~\eqref{eq: H_construction} of $H$, since $\tilde Q$ is given by the polar decomposition $Y^T\tilde Z= P\tilde Q^T$, it follows that \[\|H\|_F=\|Y - \tilde Z \tilde Q \|_F = \min_{Q\in\mathcal{O}_r}\|Y-\tilde Z Q\|_F,\] see, e.g.,~\cite[section~7.4.5]{Horn2013}. 
In general, given any $Y,\tilde Z \in \mathbb R^{n \times r}$, both of rank $r$, the minimizer $Z = \tilde Z \tilde Q$ in this problem is necessarily obtained by choosing $\tilde Q$ from the polar decomposition of $Y^T\tilde Z$ so that $Y^T Z $ is necessarily symmetric, that is, $Z$ and hence $Z - Y$ are in the horizontal space $\mathcal H_Y$. In fact, the quantity $\min_{Q \in \mathcal O_r} \| Y- \tilde Z Q \|_F$ defines a Riemannian distance between the orbits $Y\mathcal O_r$ and $\tilde Z\mathcal O_r$ in the corresponding quotient manifold; see~\cite[Proposition~5.1]{Massart2020}. 
\end{remark}

\subsection{A time interval for the factorized problem}
\label{subsec: time_interval}

We now return to the factorized problem formulation~\eqref{eq: BM}. 
Let $Y_t$ be an optimal solution of~\eqref{eq: BM} at some fixed time point~$t$ (so that $Y_t^{} Y_t^T = X^{}_t$ and $\rank Y_t^{} = r$). 
Based on the above propositions we are able to state a result on the allowed time interval $[t,t+\Delta t]$ for which the factorized problem~\eqref{eq: BM} is guaranteed to admit unique solutions on the horizontal space $\mathcal H_{Y_t}$ corresponding to the original problem~\eqref{eq: SDP}.
For this, exploiting the smoothness of the curve $t\mapsto X_t$, we first define
\begin{equation}
    \label{eq: X_lipschitz}
L:=\max_{t \in [0,T] } \| \dot X_t \|_F ,
\end{equation}
a uniform bound on the time derivative, as well as
\begin{equation}
    \label{eq: small_eig_bound}
    \lambda_r(X_t) \ge \lambda_* > 0
\end{equation}
on the smallest eigenvalue of $X_t$, are available for $t \in [0,T]$. 
Notice that the existence of such bounds is without any further loss of generality: the existence of $L$ follows from~\ref{cons: unique}, which guarantees that $X_t$ is a smooth curve, while the existence of $\lambda_*$ is guaranteed by~\ref{cons: rank}, since $X_t$ has a constant rank. 
 
\begin{theorem}
\label{thm: minimum time interval}
Let $Y_t$ be a solution of~\eqref{eq: BM} as above. 
Then for $\Delta t < \frac{2\lambda_*}{L(\sqrt{r+4} + \sqrt{r})}$ there is a unique and smooth solution curve $s \mapsto Y_s$ for the problem~\eqref{eq: BM} restricted to $\mathcal{H}_{Y_t}$ in the time interval $s \in [t, t + \Delta t]$.
\end{theorem}

\begin{proof}[Proof]
It suffices to show that for $s$ in the asserted time interval the solutions $X_s$ of~(SDP$_s$) lie in the image $\phi(Y_t+ \mathcal B_{Y_t})$.
By Proposition~\ref{prop: inverse_inject_radius}, this is the case if $\|X_s-X_t\|_F<\frac{2\lambda_r}{\sqrt{r+4} + \sqrt{r}}$, where $\lambda_r$ is the smallest eigenvalue of $X_t$.
Since
\[
    \|X_s-X_t\|_F\le\int_t^{s}\|\dot X_\tau\|_F \;d\tau\le L(s-t),
\]
and $\lambda_*\le\lambda_r$, the condition $s-t< \frac{2\lambda_*}{L(\sqrt{r+4} + \sqrt{r})}$ is sufficient. Then Proposition~\ref{prop: injectivity of phi} provides the smooth solution curve $Y_s = \phi^{-1}(X_s)$ for problem~\eqref{eq: BM}.
\end{proof}

The results of this section motivate the definition of a version of~\eqref{eq: BM} restricted to $\mathcal{H}_{Y_t}$, which we provide in the next section.
 
\section{Path following the trajectory of solutions}
\label{sec: predictor_corrector}

In this section, we present a path-following procedure for computing a sequence of approximate solutions $\{\hat Y_0,\dots,\hat Y_k,\dots,\hat Y_K \}$ at different time points that tracks a trajectory of solutions $t\mapsto Y_t$ to the Burer--Monteiro reformulation~\eqref{eq: BM}. 
From this sequence we are then able to reconstruct a corresponding sequence of approximate solutions $\hat X^{}_k=\hat Y_k^{}\hat Y_k^T$ tracking the trajectory of solutions $t\mapsto X_t$ for the full space TV-SDP problem~\eqref{eq: SDP}. 
The path-following method is based on iteratively solving the linearized KKT system. 
Given an iterate $Y_t$ on the path, we explained in the previous section how to eliminate the problem of nonuniqueness of the path in a small time interval $[t,t+\Delta t]$ by considering problem~\eqref{eq: BM} restricted to the horizontal space $\mathcal{H}_{Y_t}$.
We now need to ensure that this also guarantees that the linearized KKT system admits a unique solution. 
We show in Theorem~\ref{thm: SOSC} that this is indeed guaranteed under standard regularity assumptions on the original problem~\eqref{eq: SDP}. 
This is a remarkable fact of somewhat independent interest.

\subsection{Linearized KKT conditions and second-order sufficiency}
\label{subsec: linar_kkt}
Given an optimal solution $X_t^{}=Y_t^{}Y_t^T$ at time $t$, we aim to find a solution $X_{t+\Delta t}^{}= Y_{t+\Delta t}^{} Y_{t+\Delta t}^T$ at time $t + \Delta t$. By the results of the previous section, the next solution can be expressed in a unique way as
\begin{equation*}
    Y_{t+\Delta t} = Y_t + \Delta Y,
\end{equation*}
where $\Delta Y$ is in the horizontal space $\mathcal H_{Y_t}$, provided that $\Delta t$ is small enough. 

We define the following maps:
\begin{equation}
\label{eq: functions_def}
    \begin{aligned}
        f_{t+\Delta t}(Y)&\coloneqq 
        \langle C_{t+\Delta t}, Y Y^T \rangle,
        \\
        g_{t+\Delta t}(Y)&\coloneqq \mathcal A_{t + \Delta t}(YY^T) - b_{t+\Delta t}, 
        \\
         h_{Y_t}(Y)&\coloneqq Y^T_tY-Y^TY^{}_t.
    \end{aligned}
\end{equation}
By definition, $\Delta Y \in \mathcal H_{Y_t}$ if and only if $h_{Y_t}(\Delta Y) = 0$. 
For symmetry reasons we use the equivalent condition $h_{Y_t}(Y_t + \Delta Y) = 0$ (which reflects the fact that $Y_t + \mathcal H_{Y_t}$ is actually a linear space). 

To find the new iterate $Y_{t+\Delta t}$ we hence consider the problem
\begin{equation}
\label{eq: BM_projected}
\tag{\mbox{BM${}_{Y_t, t+\Delta t}$}}
    \begin{aligned}
        &\min_{Y\in\mathbb{R}^{n\times r}} && f_{t+\Delta t}(Y)
        \\
        &\text{\ \ \ s.t.}&& g_{t+\Delta t}(Y)= 0
        \\
        &&&h_{Y_t}(Y)=0.
    \end{aligned}
\end{equation}
This is a quadratically constrained quadratic problem whose Lagrangian is
\begin{equation}
    \label{eq: lagrangian}
    \mathcal L_{Y_t,t+\Delta t}(Y,\lambda,\mu):= f_{t+\Delta t}(Y)-\langle \lambda, g_{t+\Delta  t}(Y)\rangle- \langle \mu, h_{Y_t}(Y) \rangle 
\end{equation}
with multipliers $\lambda \in \mathbb{R}^m$ and $\mu \in \mathbb {S}^{r}_{skew}$.
The KKT conditions of  problem~\eqref{eq: BM_projected} are
\begin{equation}
\label{eq: KKT}
    \begin{aligned} 
        \nabla_Y\mathcal L_{Y_t,t+\Delta t}(Y,\lambda,\mu) =0 
        \\
        g_{t+\Delta t}(Y) = 0
        \\
        h_{Y_t}(Y) = 0.
    \end{aligned} 
\end{equation}
 Hence,~\eqref{eq: KKT} reads explicitly as
\[
    \mathcal F_{Y_t,t+\Delta t}(Y,\lambda,\mu) \coloneqq \begin{bmatrix} 
        2C_{t+\Delta t} Y - 2\mathcal A^*_{t+\Delta t}(\lambda) Y - 2Y_t\mu 
        \\ 
        \mathcal A_{t+\Delta t}(YY^T) - b_{t+\Delta t}
        \\ 
        Y_t^T Y - Y^T Y_t
    \end{bmatrix}  
    =0.
\]
The linearization of~\eqref{eq: KKT} at $(Y_t,\lambda_t,\mu_t)$ leads to a linear system
\begin{equation}
\label{eq: linear_system_step}
    \mathcal{J}_{Y_t,t+\Delta t}(Y_t,\lambda_t,\mu_t)
    \begin{bmatrix}
        \Delta Y
        \\
        \lambda_t+\Delta \lambda
        \\
        \mu_t + \Delta \mu
    \end{bmatrix}
    =
    \begin{bmatrix}
        -\nabla_Y f_{t+\Delta t}(Y_t)
        \\
        g_{t+\Delta t}(Y_t)
        \\
        0
    \end{bmatrix},
\end{equation}
where $\mathcal{J}_{Y_t,t+\Delta t}(Y,\lambda,\mu)$ denotes the derivative of $\mathcal F_{Y_t,t+\Delta t}$ at $(Y,\lambda,\mu)$.
Note that it actually does not depend on~$\mu$, but we will keep this notation for consistency. 
As a linear operator on $\mathbb R^{n \times r} \times \mathbb R^m \times \mathbb S^r_{skew}$, $\mathcal{J}_{Y_t,t+\Delta t}(Y,\lambda,\mu)$ can be written in block matrix notation as follows,
\[
    \mathcal{J}_{Y_t,t+\Delta t}(Y,\lambda,\mu):=
    \begin{bmatrix}
        \nabla^2_Y\mathcal{L}_{Y_t,t+\Delta t}(\lambda)&& -g'_{t+\Delta t}(Y)^*&&-h_{Y_t}^*
        \\
        -g'_{t+\Delta t} (Y)&&0&&0
        \\
        -h_{Y_t}&&0&&0
    \end{bmatrix},
\] 
where from \eqref{eq: functions_def} and \eqref{eq: lagrangian} one derives
\begin{align*}
    \nabla^2_Y\mathcal{L}_{Y_t,t+\Delta t}&:H \mapsto 2(C_{t+\Delta t}^{}-\mathcal A^*_{t+\Delta t}(\lambda))H,
    \\ 
    g'_{t+\Delta t}(Y)&:H\mapsto\mathcal A_{t+\Delta t}(YH^T+HY^T),
    \\
    h_{Y_t}&:H\mapsto Y_t^TH-H^TY_t,
    \\
    g'_{t+\Delta t}(Y)^*&:\lambda\mapsto2\mathcal{A}^*_{t+\Delta t}(\lambda)Y,
    \\
    h^*_{Y_t}&:\mu\mapsto 2Y_t\mu.
\end{align*} 
For later reference, observe that as a bilinear form $\nabla^2_Y\mathcal{L}_{Y_t,t+\Delta t}$ reads
\[
\nabla^2_Y\mathcal{L}_{Y_t,t+\Delta t}(\lambda)[H,H]=2\operatorname{trace}(H^T (C_{t+\Delta t}^{}-\mathcal A^*_{t+\Delta t}(\lambda))H).
\]
Solving~\eqref{eq: linear_system_step} for obtaining updates $(Y_t + \Delta Y,\lambda_t+\Delta \lambda,\mu_t + \Delta \mu)$ is equivalent to applying one step of Newton's method to the KKT system~\eqref{eq: KKT} (Lagrange--Newton method).

Our aim in this subsection is to show that for $\Delta t$ small enough the system~\eqref{eq: linear_system_step} is uniquely solvable when $(Y_t,\lambda_t)$ is a KKT-pair for the overparametrized problem~\eqref{eq: BM}. Since the system is continuous in $\Delta t$, we can do that by showing that it admits a unique solution for $\Delta t=0$. This corresponds to proving second-order sufficient conditions for the optimality of problem~\eqref{eq: BM_projected} for $\Delta t=0$. 
Interestingly, it is possible to relate this to standard regularity hypotheses on the original semidefinite problem~\eqref{eq: SDP}. For this we first need a uniqueness statement on the Lagrange multiplier $\lambda_t$.

\begin{lemma}
\label{lem: unique_lagr_mult}
Given an optimal solution $X_t^{}=Y_t^{}Y_t^T$ to~\eqref{eq: SDP}, suppose that $X_t$ is a unique (see consequence \ref{cons: unique}), primal nondegenerate (see Definition~\ref{def: non_deg} and assumption \ref{ass: non_deg}) solution. 
Then there is a unique optimal Lagrangian multiplier $\lambda_t$ for~\eqref{eq: BM} independent of the choice of $Y_t$ in the orbit $Y_t \mathcal O_r$. 
Moreover, $Z(\lambda^{}_t)=C^{}_t-\mathcal A_t^*(\lambda^{}_t)$ is the unique dual solution to~\eqref{eq: DSDP}.
\end{lemma}
\begin{proof} 
We start by recalling that the optimal set for~\eqref{eq: BM} coincides with $Y_t\mathcal{O}_r$. Since the KKT conditions for~\eqref{eq: BM} are just
\[
    \nabla_Y f_t(Y) - \nabla_Y \langle \lambda, g_t(Y)\rangle = 2(C_t-\mathcal A_t^*(\lambda))Y = 0
\]
(and $g_t(Y) = 0$), 
the set of all optimal dual multipliers for~\eqref{eq: BM} is
\[
     \{\lambda \colon (C_t-\mathcal A_t^*(\lambda))Y_tQ = 0, Q\in\mathcal{O}_r\}
     =\{\lambda \colon (C_t-\mathcal A_t^*(\lambda))Y_t = 0\}
\]
To show that this set is a singleton, it suffices to prove that the homogeneous equation $\mathcal A^*_t(\lambda) Y^{}_t=0$ has only the zero solution.
By~\eqref{eq: primal_non_deg}, primal nondegeneracy for $X_t$ can read as
\[
    \operatorname{im}\mathcal A^*_t\cap\mathcal T^\perp_{X_t}=\{0\},
\]
where $\mathcal T^\perp_{X_t}=\{M\in\mathbb{S}^{n }\mid MX_t=0\}$. Noticing that $\mathcal A^*_t(\lambda) Y_t=0$ implies $A^*_t(\lambda)\in \operatorname{im}(\mathcal A^*_t)\cap\mathcal T^\perp_{X_t}$, we get that $\mathcal A^*_t(\lambda)=0$ and thus $\lambda=0$ since $\mathcal A_t^*$ is injective by assumption~\ref{ass: licq}. 
To prove the second statement, observe that by primal nondegeneracy~\eqref{eq: DSDP} has a unique solution $Z(w_t)$ corresponding, by assumption~\ref{ass: licq}, to a unique dual multipliers vector $w_t$ (see Theorem~7 in~\cite{Alizadeh1997}). Furthermore, $Z(w_t)$ satisfies $Z(w_t)X_t=\left(C^{}_{t}-\mathcal A^*_t(w_t)\right) Y^{}_tY_t^T=0$ by~\eqref{eq: KKT_SDP}. 
Since $Y_t$ has full column rank if $r$ is chosen equal to $r^* = \rank X_t$, this implies that $\left(C^{}_{t}-\mathcal A^*_t(w^{}_t)\right) Y^{}_t=0$. From the first statement it then follows that $w_t=\lambda_t$. 
\end{proof}
  
We can now state and prove the main result of this subsection.

\begin{theorem}
\label{thm: SOSC}
Let $(X^{}_t=Y_t^{}Y_t^T,Z^{}_t)$ be a strictly complementary (see Definition~\ref{def: strict_compl}) optimal primal-dual pair of solutions to~\eqref{eq: SDP}-\eqref{eq: DSDP} such that $X_t$ is a primal nondegenerate solution. Let $\lambda_t$ be the unique corresponding Lagrange multiplier for~\eqref{eq: BM} according to Lemma~\ref{lem: unique_lagr_mult}. Then the triple $(Y_t,\lambda_t,\mu_t=0)$ is a KKT triple for~\eqref{eq: BM_projected} at $\Delta t = 0$ (that is, $\mathcal F_{Y_t,t}(Y_t,\lambda_t,0) = 0$) and fulfills the second-order sufficient conditions: 
 \begin{equation}
 \label{eq:SOSC}
     \nabla^2_Y\mathcal{L}_{Y_t,t}(\lambda_t)[H,H]=\operatorname{trace}(H^T(C^{}_t-\mathcal{A}^*_t(\lambda_t))H)>0
 \end{equation}
 for all 
 $H\in\mathbb{R}^{n\times r}\setminus\{0\}$ satisfying 
 $\mathcal A^{}_{t}(Y^{}_tH^T+HY_t^T)=0$ and $Y_t^TH-HY_t^T=0$. 
In particular, $\mathcal{J}_{Y_t,t}(Y_t,\lambda_t,0)$ is invertible. 
\end{theorem}

\begin{proof}
Since $(C_t-\mathcal A^*(\lambda_t))Y_t = Z(\lambda_t) Y_t =  0$ by the KKT conditions for~\eqref{eq: BM} and $h_{Y_t}(Y_t) = 0$, it is obvious that $\mathcal F_{Y_t,t}(Y_t,\lambda_t,0) = 0$. It is well-known that the linearized KKT system~\eqref{eq: linear_system_step} admits a unique solution if (and only if) the second-order sufficient conditions~\eqref{eq:SOSC} hold; see e.g.,~\cite[ Lemma 16.1]{Nocedal2006}. Since $(X_t
, Z(\lambda_t))$ is an optimal solution for the original primal-dual pair of
SDPs, and it hence satisifies the second-order necessary conditions for optimality (that is, $Z(\lambda_t)\succeq0$),~\eqref{eq:SOSC} holds with~``$\geq$''.
Assume that 
\[
\operatorname{trace}(H^T(C^{}_t-\mathcal{A}_t^*(\lambda^{}_t))H)=\operatorname{trace}(H^TZ(\lambda^{}_t)H)=0
\]
for some $H\in\mathbb{R}^{n\times r}$ satisfying $\mathcal A_{t}(Y_tH^T+HY_t^T)=0$ and $Y_t^TH-H^TY^{}_t=0$.
Since $Z_t = Z(\lambda_t)$ is positive semidefinite, the columns of $H$ must belong to the kernel of $Z(\lambda_t)$. 
By strict complementarity they hence belong to the column space of $X_t$, which is equal to the column space of $Y_t$. 
Therefore $H=Y_tP$ for some matrix $P\in\mathbb{R}^{r\times r}$.
Consider now the matrix
 \begin{equation*}
     \label{eq: SOSC_contra}
      \tilde  X=X^{}_t+s(Y^{}_tH^T+HY_t^T)= Y^{}_t[I^{}_r+s(P^T+P)]Y_t^T,
 \end{equation*}
depending on a real parameter $s$. 
Clearly, $\mathcal{A}_t(\tilde X)=b_t$ and, for nonzero $|s|$ small enough, $\tilde X$ is positive semidefinite. Furthermore, for a suitable choice of the sign of $s$, we have $\langle C_t,\tilde X\rangle\le\langle C_t,X_t\rangle$. Since $X_t$ is the unique solution of~\eqref{eq: SDP}, this implies $\tilde X=X_t$ and thus $Y^{}_tH^T+HY_t^T$ must be zero. 
Since $H\in\mathcal{H}_{Y_t}$ Proposition~\ref{prop: injectivity Dphi} yields $H=0$, and this completes the proof.
\end{proof}
\begin{corollary}
Let the assumptions of Theorem~\ref{thm: SOSC} be satisfied. Then for $\Delta t>0$ small enough (and depending on $Y_t$) system~\eqref{eq: linear_system_step}, that is, operator $\mathcal{J}_{Y_t,t+\Delta t}(Y_t,\lambda_t,0)$, is invertible.
\end{corollary}
Clearly, this is only a qualitative result. An upper bound for feasible $\Delta t$ could be expressed in terms of the spectral norm of the inverse of $\mathcal{J}_{Y_t,t}(Y_t,\lambda_t,0)$ using perturbation arguments. This would require a lower bound on the absolute value of the eigenvalues of $\mathcal{J}_{Y_t,t}(Y_t,\lambda_t,0)$. In this context, we should clarify that the eigenvalues, and hence also the condition number of $\mathcal{J}_{Y_t,t + \Delta t}(Y_t,\lambda_t,0)$ (for sufficiently small $\Delta t$ as above), do not depend on the particular choice of $Y_t$ in the orbit $Y_t \mathcal O_r$. This is obviously also relevant from a practical perspective. To see this, note that as a bilinear form (on $\mathbb R^{n \times r} \times \mathbb R^m \times \mathbb S^r_{skew}$) $\mathcal{J}_{Y_t,t+\Delta t}(Y,\lambda,\mu)$ reads
\begin{equation*}
\label{eq: J as bilinear form}
\begin{aligned}
&\mathcal{J}_{Y_t,t+\Delta t}(Y,\lambda,\mu)[(H,\Delta \lambda,\Delta \mu),(H,\Delta \lambda,\Delta \mu)]
\\ {}={} &\operatorname{trace}(H^T (C_{t+\Delta t}^{}-\mathcal A^*_{t+\Delta t}(\lambda))H) - 2\langle \Delta \lambda, \mathcal A_{t+\Delta t}(YH^T+HY^T)\rangle - 2 \langle \Delta \mu , Y_t^TH-H^TY_t \rangle.
\end{aligned}
\end{equation*}
For any fixed $Q \in \mathcal O_r$ one therefore has
\[
\begin{aligned}
&\mathcal{J}_{Y_t,t+\Delta t}(Y_t,\lambda_t,0)[(H,\Delta \lambda,\Delta \mu),(H,\Delta \lambda,\Delta \mu)] \\
{}={} &\mathcal{J}_{Y_tQ,t+\Delta t}(Y_tQ,\lambda_t,0)[\mathcal T_Q (H,\Delta \lambda,\Delta \mu),\mathcal T_Q (H,\Delta \lambda,\Delta \mu)]
\end{aligned}
\]
with the unitary linear operator $\mathcal T_Q (H,\Delta \lambda,\Delta \mu) = (HQ, \Delta \lambda, Q^T \Delta \mu Q)$ on $\mathbb R^{n \times r} \times \mathbb R^m \times \mathbb S^r_{skew}$. It follows that $\mathcal{J}_{Y_t,t+\Delta t}(Y_t,\lambda_t,0)$ and $\mathcal{J}_{Y_tQ,t+\Delta t}(Y_tQ,\lambda_t,0)$ have the same eigenvalues.  

However, our proof of Theorem~\ref{thm: SOSC} is by contradiction and hence does not provide an obvious lower bound on the radius of invertibility of $\mathcal J_{Y_t,t}(Y_t,\lambda_t,0)$. Here we do not intend to investigate this in more depth. In the error analysis conducted later we will essentially assume to have such a bound available (cf.~Lemma~\ref{lem: bounded inverses}).

\subsection{A path-following predictor-corrector algorithm}
\label{subsec: path_foll_algorithm}

We now thoroughly describe the path-following predictor-corrector algorithm that we propose for tracking the trajectory of solutions to~\eqref{eq: SDP}. It includes an optional adaptive step size tuning step which is based on measuring the \emph{residual} of the optimality conditions, defined as 
\begin{equation}
\label{residual}
    \operatorname{res}_t(Y,\lambda):=
    \left \|\begin{array}{c}
    2[C_{t}  - \mathcal A^*_{t}(\lambda)] Y
    \\ 
    \mathcal A_{t}(YY^T) - b_{t}
    \end{array}\right\|_{\infty}.
\tag{RES}
\end{equation} 
 The residual expresses the maximal component-wise violation of the optimality KKT conditions for the problem~\eqref{eq: BM} and is therefore a suitable error measure. Indeed (see, e.g.,~\cite[Theorems 3.1 and 3.2]{Wright2005}),
if the second-order sufficiency condition for optimality holds at $(Y_t,\lambda_t)$, then there are constants $\eta,C_1,C_2>0$ such that for all $(Y, \lambda)$ with $\|(Y, \lambda)-(Y_t,\lambda_t)\| \le \eta$ one has
\[
    C_1\|(Y, \lambda)-(Y_t,\lambda_t)\|\le\operatorname{res}_t(Y, \lambda)\le C_2\|(Y, \lambda)-(Y_t,\lambda_t)\|.
\]
Here and in the following, we we use the norm $\|(Y,\lambda)\|^2 = \| Y\|_F^2 + \| \lambda \|^2$. 

The overall procedure is displayed as Algorithm~\ref{alg: 1} below. Given a TV-SDP of the form \eqref{eq: SDP}, parameterized over a time interval $[0,T]$, the inputs are an approximate initial primal-dual solution pair $(\hat X_0, Z(\hat\lambda_0))$ to (SDP$_0$)--(D-SDP$_0$) and an initial step size $\Delta t_0$. At each iteration the current iterate is used to construct the linear system~\eqref{eq: linear_system_step}, which is then solved, returning the updates $\Delta Y$ and $\Delta\lambda$. 
The presented version of the algorithm also includes a procedure for tuning the step size that can be activated through the Boolean variable \texttt{step size\_TUNING} and is supposed to ensure that the residual threshold is satisfied at every time step. Specifically, if for a time step the threshold is violated, the step size is reduced by a factor $\gamma_1\in(0,1)$ and a more accurate solution is obtained by solving the linearized KKT system \eqref{eq: linear_system_step} for the reduced time step. On the other hand, to avoid unnecessary small steps, the step size is increased after every successful step by a factor $\gamma_2 > 1$ (but is never made larger than $\Delta t_0$).  If the step size tuning is deactivated, the algorithm just runs with the constant step size $\Delta t_0$ instead. 
Note that Algorithm \ref{alg: 1} tracks both the primal solution $X_t$ and the dual solution $Z_t=C_t-\mathcal{A}^*_t(\lambda_t)$.
 
\begin{algorithm}
\caption{Path-following predictor-corrector for~\eqref{eq: SDP} with $t\in[0,T]$}
\label{alg: 1}
        \textbf{Input:} an initial approximate primal-dual solution $(\hat X_0,Z(\hat \lambda_0))$ to (SDP$_0$)--(D-SDP$_0$)  
        \\ 
        initial step size $\Delta t_0$
        \\
        boolean variable \texttt{step size\_TUNING}
        \\
        step size tuning parameters $\gamma_1\in(0,1)$, $\gamma_2>1$
        \\ 
        residual tolerance $\epsilon>0$
        \\ 
        \textbf{Output:} solutions $\{\hat X_k\}_{k=0,\dots,K}$ to~\eqref{eq: SDP} for $t\in\{0,\dots,t_k,\dots,T\}$
        \begin{algorithmic}[1]
        \STATE $k\xleftarrow{}0$
        \STATE $t_0\xleftarrow{}0$
        \STATE $\Delta t\xleftarrow{}\Delta t_0$
        \STATE $S=\{\hat X_0\}$, $r = \rank(\hat X_0)$
        \STATE find $\hat Y_0\in\mathbb{R}^{n\times r}$ such that $\hat Y^{}_0\hat Y_0^T=\hat X^{}_0$
        \WHILE{$t_k< T $}{
        \STATE solve linear system~\eqref{eq: linear_system_step} with data $\Delta t,t_k,\hat Y_k,\hat \lambda_k$
        and obtain $\Delta Y, \Delta\lambda$
        \IF{\texttt{step size\_TUNING} \AND $\operatorname{res}_{\hat Y_k,t_{k}+\Delta t}(\hat Y_k+\Delta Y,\hat \lambda_k+\Delta\lambda)>\epsilon$}
        \STATE $\Delta t\xleftarrow{}\gamma_1\Delta t$
        \STATE go back to step 6
        \ENDIF
        \STATE $(t_{k+1},\hat Y_{k+1},\hat \lambda_{k+1} )\xleftarrow{}(t_{k}+\Delta t,\hat Y_k+\Delta Y,\hat \lambda_k+\Delta\lambda)$
        \STATE append $\hat X_{k+1} ={} \hat Y^{}_{k+1}\hat Y_{k+1}^T$ to $S$
        \IF{\texttt{step size\_TUNING}}
        \STATE $\Delta t\xleftarrow{}\min(T-t_{k+1},\gamma_2\Delta t,\Delta t_0)$ 
        \ELSE
        \STATE $\Delta t\xleftarrow{}\min(T-t_{k+1},\Delta t)$
        \ENDIF
        \STATE  $k\xleftarrow{}k+1$}
        \ENDWHILE
        \RETURN $S$ 
        \end{algorithmic} 
\end{algorithm}

\subsection{Error analysis}
\label{subsec: err_analysis}

We investigate the algorithm without step size tuning. 
The main goal of the following error analysis to show that the computed $(\hat X_k, \hat \lambda_k )$, where $\hat X_k = \hat Y_k \hat Y_k^T$, remain close to the exact solutions $(X_{t_k},\lambda_{t_k})$, if properly initialized. 
The logic of the proof is similar to standard path following methods based on Newton's method, e.g.~\cite{Diehl2012}. The specific form of our problem requires some additional considerations that allow for more precise quantitative bounds depending on the problem constants. 

Throughout this section, $(X^{}_t=Y_t^{}Y_t^T,Z^{}_t)$ is an optimal primal-dual pair of solutions to \eqref{eq: SDP}--\eqref{eq: DSDP} satisfying the five assumptions \ref{ass: strict_feas}--\ref{ass: cont_diff}, so that it is strictly complementary (see Definition~\ref{def: strict_compl}) and such that $X_t$ is primal nondegenerate.
Notice that the choice of factor $Y_t$ can be arbitrary, since it does not affect any of the subsequent statements.
In Lemma~\ref{lem: unique_lagr_mult} and its proof, we have seen that for every $X_t$ the unique Lagrange multiplier $\lambda_t$ satisfies $Z_t =  C_t - \mathcal A^*_t(\lambda_t)$, that is,
\[
    \lambda_t = (A^*_t)^\dagger(Z_t - C_t)
\]
with $(A^*_t)^\dagger$ being the pseudo-inverse of $A^*_t$. By assumption \ref{ass: cont_diff}, $C_t$ and $\mathcal A^*_t$ depend smoothly on $t$ and so does $(\mathcal{A}_t^*)^\dagger$, since $\mathcal A_t$ is surjective for all $t$ by assumption \ref{ass: licq}. Also, by Theorem~\ref{thm:single_valued_differentiable}, $t \mapsto Z_t$ is smooth. Therefore the curve $t \mapsto \lambda_t$ is smooth. Since the algorithm operates in the $(Y,\lambda)$ space, our implicit goal is to show that the iterates stay close to the set 
\[
\mathcal C := \{ (Y_t,\lambda_t) \mid \text{$(Y_t^{} Y_t^T,Z(\lambda_t))$ is an optimal primal-dual pair to~\eqref{eq: SDP}--\eqref{eq: DSDP}}, \ t \in [0,T] \}
\]
containing the optimal primal-dual trajectories in the Burer--Monteiro factorization.

\begin{lemma}
The set $\mathcal C$ is compact.
\end{lemma}

\begin{proof}
As the curve $t \mapsto \lambda_t$ is continuous, it suffices to prove that the set $\mathcal C_Y = \{ Y_t \mid  t \in [0,T]  \}$ is compact. Since $\| Y_t \|_F = \sqrt{\tr(X_t)}$ and $t \mapsto X_t$ is smooth, it is bounded. To see that the set is closed, let $(Y_n) \subset \mathcal C_Y$ be a convergent sequence with limit $Y$ such that $Y_n^{} Y_n^T = X_{t_n}$ for some $t_n \in [0,T] $. By passing to a subsequence, we can assume $t_n \to t \in [0,T] $. Then obviously $X_t = Y Y^T$, which shows that $Y$ is in the set.
\end{proof}

We consider the norm on $\mathbb R^{n \times r} \times \mathbb R^m \times \mathbb S^r_{skew}$ defined by $\|(Y,\lambda,\mu)\|^2 = \| Y \|_F^2 + \| \lambda \|^2 + \| \mu \|_F^2$. The induced operator norm is denoted $\| \cdot \|_{op}$. 

\begin{lemma}\label{lem: bounded inverses}
There exists a constant $m > 0$ such that
\begin{equation}\label{eq: bounded inverse}
\| \mathcal J_{Y_t,t}(Y_t,\lambda_t,0)^{-1} \|_{op} \le \frac{1}{m}
\end{equation}
for all $(Y_t,\lambda_t) \in \mathcal C$.
\end{lemma}

\begin{proof}
On its open domain of definition, the map $(Y,\lambda) \mapsto \| \mathcal J(Y,\lambda,0)^{-1}\|_{op}$ is continuous. By Theorem~\ref{thm: SOSC}, the compact set $\mathcal C$ is contained in that domain. Therefore, $\| \mathcal J(Y,\lambda,0)^{-1}\|_{op}$ achieves its maximum on $\mathcal C$.
\end{proof}

\begin{lemma}
For any $t \in [0,T] $ and $\hat Y \in \mathbb R^{n \times r}$, the mapping $(Y,\lambda,\mu) \mapsto \mathcal J_{\hat Y,t}(Y,\lambda,\mu)$ is Lipschitz continuous in the operator norm on $\mathbb R^{n \times r} \times \mathbb R^m \times \mathbb S^r_{skew}$. Specifically,
\[
\| \mathcal J_{\hat Y,t}(Y_1,\lambda_1,\mu_1) - \mathcal J_{\hat Y,t}(Y_2, \lambda_2,\mu_2 ) \|_{op} \le 12 \sqrt3 \| \mathcal A_t \| \| (Y_1, \lambda_1, \mu_1)  - (Y_2,\lambda_2,\mu_2 )\|
\]
for all $(Y_1,\lambda_1,\mu_1)$ and $(Y_2,\lambda_2,\mu_2)$, 
where $\| \mathcal A_t \|$ is the operator norm of $\mathcal A_t$.
\end{lemma}

\begin{proof}
It follows from~\eqref{eq: J as bilinear form} that as a bilinear form one has
\begin{align*}
 &(\mathcal J_{\hat Y,t}(Y_1,\lambda_1,\mu_1) - \mathcal J_{\hat Y,t}(Y_2,\lambda_2,\mu_2 ))[(H,\Delta \lambda,\Delta \mu),(H,\Delta \lambda,\Delta \mu)] \\
{}={} & \tr(H^T \mathcal A^*_t(\lambda_2 - \lambda_1) H) - 2(\Delta \lambda)^T \mathcal A_t((Y_1 - Y_2) H^T + H(Y_1 - Y_2)^T) \\ 
{}\le{} & \| \mathcal A_t \| \| \lambda_1 - \lambda_2 \| \| H \|_F^2 + 4 \| \mathcal A_t \| \| Y_1 - Y_2 \|_F \| H \|_F \| \Delta \lambda \| \\ 
{}\le{}& (\| \mathcal A_t \| \| \lambda_1 - \lambda_2 \| + 4 \| \mathcal A_t \| \| Y_1 - Y_2 \|_F)(\| H \|_F + \| \Delta \lambda \| + \| \Delta \mu \|_F)^2 \\
{} \le{} & 4 \| \mathcal A_t \|(\| Y_1 - Y_2\|_F + \| \lambda_1 -  \lambda_2 \|+\| \mu_1 -  \mu_2 \|)(\| H \|_F + \| \Delta \lambda \| + \| \Delta \mu \|_F)^2 \\
{}\le{} &12 \sqrt3 \| \mathcal A_t \| \| (Y_1, \lambda_1, \mu_1)  - (Y_2,\lambda_2,\mu_2 )\| \| (H,\Delta \lambda, \Delta \mu) \|^2.
\end{align*}
This proves the claim.  
\end{proof}

Since $t \mapsto \mathcal A_t$ is assumed to be continuous, the constant $M = \max_{t \in [0,T] } 12 \sqrt3 \| \mathcal A_t \|$ satisfies the uniform Lipschitz condition
\begin{equation}\label{eq: Lipschitz continuity}
\| \mathcal J_{\hat Y,t}(Y_1,\lambda_1,\mu_1) - \mathcal J_{\hat Y,t}(Y_2,\lambda_2,\mu_2) \|_{op} \le M \| (Y_1, \lambda_1, \mu_1)  - (Y_2,\lambda_2,\mu_2 ) \|
\end{equation}
for all $(Y_1,\lambda_1,\mu_1)$ and $(Y_2,\lambda_2,\mu_2)$, independent of the choice of $\hat Y \in \mathbb R^{n \times r}$. In what follows, we proceed with using~\eqref{eq: Lipschitz continuity} and~\eqref{eq: bounded inverse}, without further investigating the sharpest possible bounds.

In addition, let $\lambda_r(X_t) \ge \lambda_*> 0$ be a uniform lower bound on the smallest positive eigenvalue as in~\eqref{eq: small_eig_bound}. Furthermore, we now also assume a uniform upper bound
\[
\| Y_t \|_2 = \sqrt{ \lambda_1(X_t)} \le \sqrt{\Lambda_*}.
\]
on the spectral norm of $Y_t$. Finally, let $\| \dot X_t \|_F \le L$ as in~\eqref{eq: X_lipschitz} and since the curve $t \mapsto \lambda_t$ is smooth, the constant
\begin{equation}
    \label{eq: dual_lipschitz}
    K := \max_{t \in [0,T] } \| \dot \lambda_t \|
\end{equation}
is also well-defined.

With the necessary constants at hand, we are now in the position to state our main result on the error analysis. The following theorem shows that we can bound the distance between the iterates of Algorithm~\ref{alg: 1} and the set of solutions to~\eqref{eq: BM} provided the initial point is close enough to the set of initial solutions and the step size $\Delta t$ is small enough. Here we employ again the natural distance measure $\min_{Q \in \mathcal O_r} \| \hat Y - Y Q \|_F$ between the orbits $\hat Y \mathcal O_r$ and $ Y\mathcal O_r$, cf.~Remark~\ref{rem: polar_decomposition}.

\begin{theorem}
Let $\delta >0$ and $\Delta t > 0$ be small enough such that the following three conditions are satisfied:
\begin{gather}
\label{eq: condition 1}
(2 \sqrt{\Lambda_*} + \delta)\delta + L \Delta t < \frac{2 \lambda_*}{\sqrt{r+4} + \sqrt{r}},
\\
\label{eq: condition 2}
\delta < \frac{2}{3}\frac{m}{M},
\\
\label{eq: condition 3}
\left[ \frac{1}{\lambda_*} ((2 \sqrt{\Lambda_*} + \delta) \delta + L \Delta t)^2 + \sqrt{r}(2 \sqrt{\Lambda_*} + \delta) \delta + L \Delta t\right]^2 +
(\delta + K \Delta t)^2 \le\frac{2}{3} \frac{m}{M} \delta.
\end{gather}
Assume for the initial point $(\hat Y_0,\hat \lambda_0)$ that
\begin{equation}
\label{eq: initial_error}
\min_{Q \in \mathcal O_r} \| (\hat Y_0,\hat \lambda_0) - (Y_0 Q,  \lambda_0) \| \le \delta.
\end{equation}
Then Algorithm~\ref{alg: 1} is well-defined and for all $t_{k+1} = t_k + \Delta t$ the iterates satisfy
\[
\min_{Q \in \mathcal O_r}\| (\hat Y_k , \hat \lambda_k)- (Y_{t_k}Q, \lambda_{t_k}) \| \le \delta. 
\]
It then holds that
\[
\| \hat X_k - X_{t_k} \|_F \le (2 \sqrt{\Lambda_*} + \delta) \delta
\]
for all $t_k$.
\end{theorem} 
Notice that the left side of~\eqref{eq: condition 3} is $O(\delta^2 + \Delta t^2)$ for $\delta, \Delta t \to 0$, whereas the right side is only $O(\delta)$. Therefore for $\delta$ and $\Delta t$ small enough,~\eqref{eq: condition 3} will be satisfied. Furthermore, a sufficient condition for~\eqref{eq: initial_error} to hold is that \[\|\hat \lambda_0 - \lambda_0 \| \le \frac{\delta}{\sqrt{2}}\] and \[\| \hat X_0 - X_0 \|_F \le \frac{\sqrt{r\lambda_*^2+2\sqrt{2}\delta\lambda_*}-\sqrt{r\lambda_*^2}}{2},\] which easily follows from~\eqref{eq: inverse estimate}. 

\begin{proof}
We will investigate one step of the algorithm and apply an induction hypothesis that at time point $t = t_k$ there exists $(Y_t,\lambda_t) \in \mathcal C$ satisfying
\[
\| (\hat Y_t ,\hat \lambda_t ) - (Y_t,\lambda_t) \|_F 
\le \delta.
\]
We aim to show that for sufficiently small $\delta > 0$ and $\Delta t >0$ the next iterate $(\hat Y_{t+\Delta t},\hat \lambda_{t+\Delta t})$ in the algorithm is well-defined and satisfies the same estimate
\[
\| (\hat Y_{t+\Delta t} ,\hat \lambda_{t+\Delta t}) - (Y_{t+\Delta t}, \lambda_{t+\Delta t}) \|_F \le \delta
\]
with an exact solution $(Y_{t+\Delta t}, \lambda_{t+\Delta t}) \in \mathcal C$. The proof of the theorem then follows by induction over the steps in the algorithm.

We first claim that there exists an exact solution $Y_{t + \Delta t}$ in the horizontal space of $\hat Y_t$, that is, $X^{}_{t+\Delta t} = Y^{}_{t + \Delta t} Y_{t+\Delta t}^T$ and $h_{\hat Y_t}(Y_{t+\Delta t}) = 0$. Indeed, using~\eqref{eq: condition 1} we have
\begin{align*}
\| \hat X_t - X_t \|_F &= \| (\hat Y_t - Y_t) \hat Y^T_t + Y_t(\hat Y_t - Y_t)^T  \|_F \\ &\le (\| Y_t \|_2 +  \| \hat Y_t \|_2) \delta 
\le (2 \sqrt{\Lambda_*} + \delta) \delta < \frac{2 \lambda_*}{\sqrt{r+4} + \sqrt{r}} - L \Delta t.
\end{align*}

This yields
\[
\| \hat X_t - X_{t+\Delta t} \|_F \le \| \hat X_t - X_t \|_F + \| X_t - X_{t+\Delta t} \|_F < \frac{2 \lambda_*}{\sqrt{r+4} + \sqrt{r}}.
\]
Thus, Proposition~\ref{prop: inverse_inject_radius} states the existence of $Y_{t+\Delta t}$ as desired. We note for later use that by~\eqref{eq: inverse estimate} it satisfies
\begin{align}
\| \hat Y_t - Y_{t+\Delta t} \|_F &\le \frac{1}{\lambda_*} \| \hat X_t - X_{t+\Delta t} \|_F^2 + \sqrt{r}\| \hat X_t - X_{t+\Delta t} \|_F \notag \\
&\le 
\frac{1}{\lambda_*} (\| \hat X_t - X_{t} \|_F + L \Delta t)^2 + \sqrt{r}\| \hat X_t - X_{t} \|_F + L \Delta t \label{eq: est} \\
&\le \frac{1}{\lambda_*} ((2 \sqrt{\Lambda_*} + \delta) \delta + L \Delta t)^2 + \sqrt{r}(2 \sqrt{\Lambda_*} + \delta) \delta + L \Delta t. \notag
\end{align}

The matrix $Y_{t+\Delta t}$ is an exact solution of  (BM$_{t+\Delta t}$), and by Theorem~\ref{thm: SOSC} there is a unique Lagrange multiplier $\lambda_{t+\Delta t}$ such that $\mathcal F_{\hat Y_t,t+\Delta t}(Y_{t+\Delta t}, \lambda_{t+\Delta t}, 0) = 0$. 
By construction, the next iterate $(\hat Y_{t+\Delta t}, \hat \lambda_{t + \Delta t}, \hat \mu_{t+\Delta t})$ in the algorithm is obtained from one step of the Newton method for solving this equation with starting point $(\hat Y_t,\hat \lambda_t, 0)$. 
In light of~\eqref{eq: bounded inverse} and~\eqref{eq: Lipschitz continuity}, standard results (e.g. Theorem 1.2.5 in \cite{Nesterov2018}) on the Newton method yield that under the condition 
\begin{equation*}
\label{eq: condition epsilon}
    \|(\hat Y_t,\hat \lambda_t,0) - (Y_{t+ \Delta t},\lambda_{t+\Delta t},0) \|_F \le \varepsilon < \frac{2}{3}\frac{m}{M}
\end{equation*}
one step of the method is well-defined, i.e., $\mathcal J_{\hat Y_t,t+\Delta t}(\hat Y_t, \hat \lambda_t,0)$ is invertible, and satisfies
\[
    \| (\hat Y_{t+\Delta t}, \hat \lambda_{t + \Delta t}, \hat \mu_{t+\Delta t}) - (Y_{t+ \Delta t},\lambda_{t+\Delta t},0) \|_F \le \frac{3}{2}\frac{M}{m} \| (\hat Y_t,\hat \lambda_t,0) - (Y_{t+ \Delta t},\lambda_{t+\Delta t},0) \|_F^2.
\]
In particular, using $\varepsilon=\left( \frac{2}{3} \frac{m}{M} \delta \right)^{1/2}$ would give the desired result
\begin{equation*}
\label{eq: Newton estimate}
    \| (\hat Y_{t+\Delta t}, \hat \lambda_{t + \Delta t}) - (Y_{t+ \Delta t},\lambda_{t+\Delta t}) \|_F \le \frac{3}{2}\frac{M}{m} \varepsilon^2=\delta.
\end{equation*} 
Therefore, we need to ensure that
\begin{equation*} 
    \| (\hat Y_t,\hat \lambda_t) - (Y_{t+ \Delta t},\lambda_{t+\Delta t}) \|_F \le  \left( \frac{2}{3} \frac{m}{M} \delta \right)^{1/2} < \frac{2}{3}\frac{m}{M}
\end{equation*}
is satisfied. Here the second inequality is just condition~\eqref{eq: condition 2}.
We now show that~\eqref{eq: condition 3} is a sufficient condition for the first inequality. 
Clearly, using~\eqref{eq: dual_lipschitz}, we have
\[
    \| \hat \lambda_t - \lambda_{t+\Delta t} \|^2 \le (\| \hat \lambda_t - \lambda_t \| + K \Delta t)^2 \le (\delta + K \Delta t)^2.
\]
Together with~\eqref{eq: est} this gives
\[
    \| \hat Y_t - Y_{t+\Delta t} \|_F^2 + \| \hat \lambda_t - \lambda_{t+\Delta t} \|^2 \\\le \left[ \frac{1}{\lambda_*} [(2 \sqrt{\Lambda_*} \! + \! \delta) \delta \! + \! L \Delta t]^2\! + \!\sqrt{r}(2 \sqrt{\Lambda_*} \! + \! \delta) \delta \! + \! L \Delta t\right]^2\hspace{-0.5em}+ (\delta + K \Delta t)^2.
\]
Now~\eqref{eq: condition 3} ensures the desired estimate for the right-hand side and the proof is completed.
\end{proof}

\section{Numerical experiments on Time-Varying Max Cut}
\label{sec: num_exp}

In this section, we compare the tracking of the trajectory of solutions to TV-SDP via Algorithm~\ref{alg: 1} with interior-point methods (IPMs) used to track the same trajectory by solving the problem at discrete time points. In our experiments, we used the implementation of the homogeneous and self-dual algorithm~\cite{Andersen2003, Freund2006} from the MOSEK Optimization Suite, version~9.3~\cite{MOSEK}.  
Furthermore, in order to provide a comparison with an alternative warm-start approach, we performed numerical experiments using the Splitting Conic Solver (SCS), version 3.2.2 \cite{scs}. This package implements the first-order method presented in  \cite{Donoghue2021,Donoghue2016}, which uses an operator splitting method, the alternating directions method of multipliers, to solve the homogeneous self-dual embedding.
We show the algorithm proposed in this paper can perform better, in terms of both accuracy and runtime, than repeated runs of IPM for time-invariant SDP and than the warm-started SCS.

Given a weighted graph $\mathcal{G}=(V,E)$, the Max-Cut problem is a well-known problem in graph theory. There, we wish to find a binary partition of the vertices in $V$ (also known as a cut) of maximal weight. The weight of the cut is defined as the sum of the weights of the edges in $E$ connecting the two subsets of the partition. This problem can be formulated as the following quadratically-constrained quadratic problem
\begin{equation}
\label{eq: MC}
    \tag{MC}
    \begin{aligned}
        &\max_{x\in\mathbb{R}^n} && \sum_{i,j=1}^n w_{i,j}(1-x_ix_j)\\
        &\text{\ \ s.t.}&&x_i^2=1\quad\text{for all } i\in \{1,\dots,n\},
    \end{aligned}
\end{equation} 
where $n=|V|$ is the number of vertices of the graph, $w_{i,j}$ is the weight of the edge connecting vertices $i$ and $j$, and variable $x_i\in\{1,-1\}$ takes binary values according to the subset to which vertex $i$ is assigned. 
This problem can be relaxed to an SDP of the form
\begin{equation}
\label{eq: MCR}
    \tag{MCR}
    \begin{aligned}
        &\min_{X\in\mathbb{S}^n} && \langle W, X\rangle\\
        &\text{\ \ s.t.}&&X_{i,i}=1\quad\text{for all } i\in \{1,\dots,n\}\\
        &&&X\succeq0,
    \end{aligned}
\end{equation}
where $W$ is the weights matrix whose entry $(i,j)$ is given by $w_{i,j}$, see \cite{Goemans1995}. Note that the number of constraints is equal to the size of the variable matrix.
Randomized approximation algorithms for~\eqref{eq: MC} exploiting the convex relaxation~\eqref{eq: MCR} deliver solutions with a performance ratio of $0.87$ and are known to be the best poly-time algorithms to approximately solve~\eqref{eq: MC}.

In this paper, we adopt a time-varying version of~\eqref{eq: MCR} as a benchmark, where the data matrix $W$ depends on a time parameter $t\in[0,1]\mapsto W_t\in\mathbb{S}^n$. (We point out that this differs from the recently studied variant \cite{Kao2017,henzinger2022practical} with edge insertions and deletions, which could be seen as discontinuous functions of time.)  

In our experiment, $W_t$ is obtained as a random linear perturbation of a sparse weight matrix with density $50\%$. Specifically,  
\[
W_t=W_0+tW_1,
\]
where the entries of $W_0$ are randomly generated with a normal distribution having mean and standard deviation $\mu,\sigma=10$, while the entries in $W_1$ are chosen with a normal distribution having $\mu,\sigma=1$. Both matrices have the same sparsity structure. We refer to such a problem as the \textit{time-varying max-cut relaxation} (TV-MCR), which can be thought of as a convex relaxation for a max-cut problem where the  edges weights of a given graph change over time.

All the experiments were conducted on a personal computer with a 1,6 GHz Intel Core i5 dual-core processor with 16GB RAM, using a Python implementation of our path-following algorithm. The main goal was to illustrate the potential computational benefits of our algorithm, so we did not attempt to provide the most efficient implementation. The code\footnote{\texttt{https://github.com/antoniobellon/burer-monteiro-path-following}, Eclipse Public License 2.0.} as well as the data and experimental results\footnote{\texttt{https://zenodo.org/record/7769225}} are available online. 

We performed experiments on $110$ instances of the TV-MCR problem with $n=100$ vertices and tracked the trajectory of solutions for $t\in[0,1]$. Among these samples, we included 10 instances of TV-MCR for which the rank of the solution is not constant, hence violating our assumption~\ref{ass: strict_compl}. This was done by sampling the rank (estimated with a tolerance on zero eigenvalues of~$10^{-7}$) of the solutions obtained using MOSEK over a 10-steps subdivision of the interval $[0,1]$ and selecting ten cases in which we observed a change in the rank. Using the same procedure, we checked that for the remaining 100 instances, the rank of the solution is constant along the trajectory.

First, we applied Algorithm~\ref{alg: 1} without step size adjustment, hence setting \texttt{step size\_TUNING} to \texttt{FALSE}, and using step sizes $\Delta t=0.1,0.01,0.001$, so that in each experiment 10, 100, and 1000 iterations are performed for each choice of the step size (see Figures~\ref{fig:residuals} and~\ref{fig:runtimes}).
The factor dimension~$r$ is chosen equal to the rank of an initial solution obtained using MOSEK with relative gap termination tolerances set to $10^{-14}$. Its distribution is shown in Table \ref{tab: 1}.  

\begin{table}[H]
    \centering 
    \begin{tabular}{lcccc}
        \hline
         $r$&4&5&6&7
         \\
         \hline
         \# occurences&2&39&53&6\\
         \hline
    \end{tabular}
    \caption{Distribution  of the rank over $100$ instances of the TV-MCR with $n=100$ with constant rank solution trajectory.}
    \label{tab: 1}
\end{table} 
\begin{figure}
    \centering
    \includegraphics[width=\textwidth]{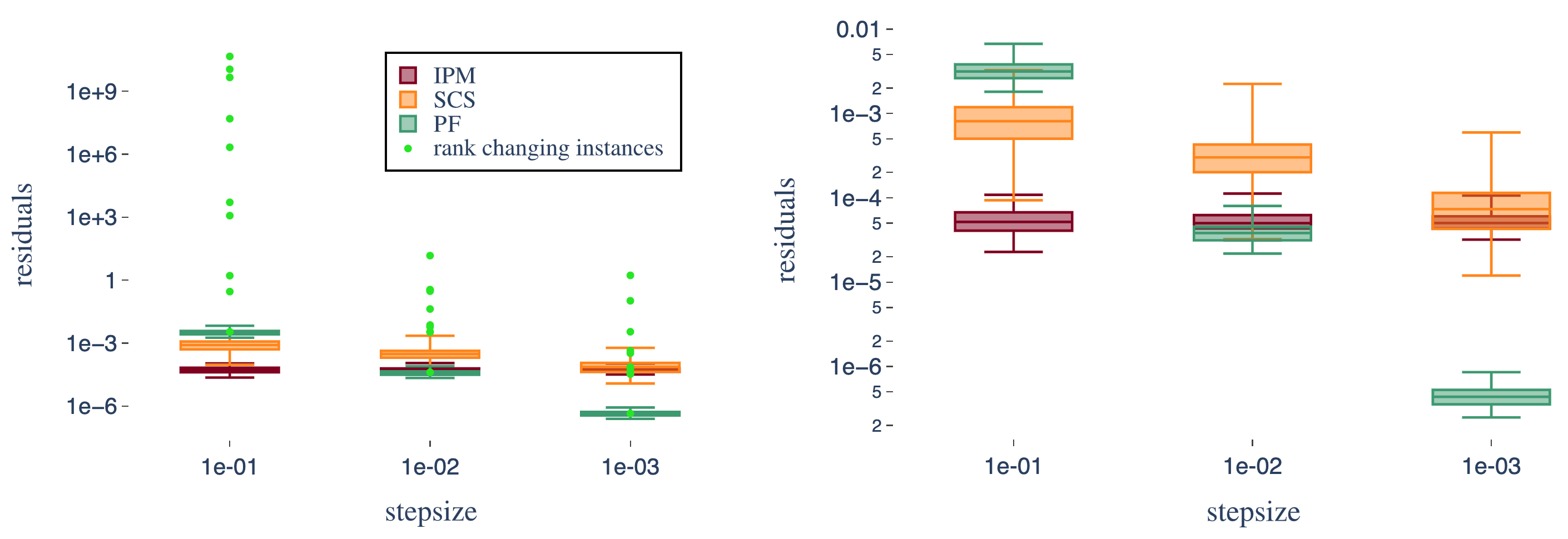}
    \caption{Distribution of the average residuals as a function of the step size using three different methods: an interior point method (IPM), in bordeaux, the splitting conic solver (SCS), in orange, and our path following (PF) algorithm, in green. The data in both plots are the same except that the left plot also shows ten rank changing instances, depicted by light green dots, which were removed in the right plot.}
\label{fig:residuals}  
\end{figure}  

 Figure~\ref{fig:residuals} depicts the distribution over 100 instances of the average residuals along the tracking of the solution on the time interval $[0,1]$, as a function of the used step size. 
For each whisker plot, the error bars span the interval from the minimum to the maximum, while the box spans the first quartile to the third quartile, with a horizontal line at the median. 

In the left plot, the light green dots correspond to the average residuals of the 10 rank-changing instances; instead, the right plot excludes these degenerate instances form the data set.
Notice that these points correspond to TV-SDP instances that do not satisfy our assumption~\ref{ass: strict_compl}.
The green plot shows the average residual obtained by tracking the solution with Algorithm~\ref{alg: 1}, the orange plot shows the average residual when the tracking is done using SCS with relative and absolute feasibility tolerances set to $10^{-7}$, warm-started with the current solution; finally, the bordeaux color plot shows the average residual when the tracking is done using MOSEK IPM \cite{MOSEK} with the relative gap termination tolerances set to $10^{-15}$.

The residual of an SDP primal-dual solution $(X,Z(\lambda))$ is defined, in analogy to~\eqref{residual}, as
\[
    \operatorname{res}_{t}(X,\lambda):=
    \left \|
    \begin{array}{c}
        2[C_{t}  - \mathcal A^*_{t}(\lambda)] X
        \\ 
        \mathcal A_{t}(X) - b_{t}
    \end{array}
    \right\|_{\infty}.
\] 
By choosing a suitable step size (in our experiments order $10^{-2}$), Algorithm~\ref{alg: 1} yields an average residual accuracy that is comparable to the one obtained using standard IPMs with very small relative gap termination tolerance. For a step size of order $10^{-3}$, our algorithm exhibits a residual precision that is 100 times more accurate than both IPM and warm-started SCS. Furthermore, as we see next, this accuracy is reached much faster with our approach.

In Figure~\ref{fig:runtimes} we plot the distributions of the runtimes of Algorithm~\ref{alg: 1} (green) as a function of the step size, as well as the distributions of the runtimes of IPM (bordeaux) used with relative gap termination tolerances $10^{-15}$ and of the warm-started SCS (orange) to track the solutions trajectory at a constant step size resolution.

Remarkably, for each step size that we tested, the mean runtime of Algorithm~\ref{alg: 1} is on average about ten times smaller then both SCS and MOSEK IPM, indicating competitive computational performances of our algorithm.

\bigskip
\begin{figure}[H]
    \centering  \includegraphics[height=0.3\textwidth]{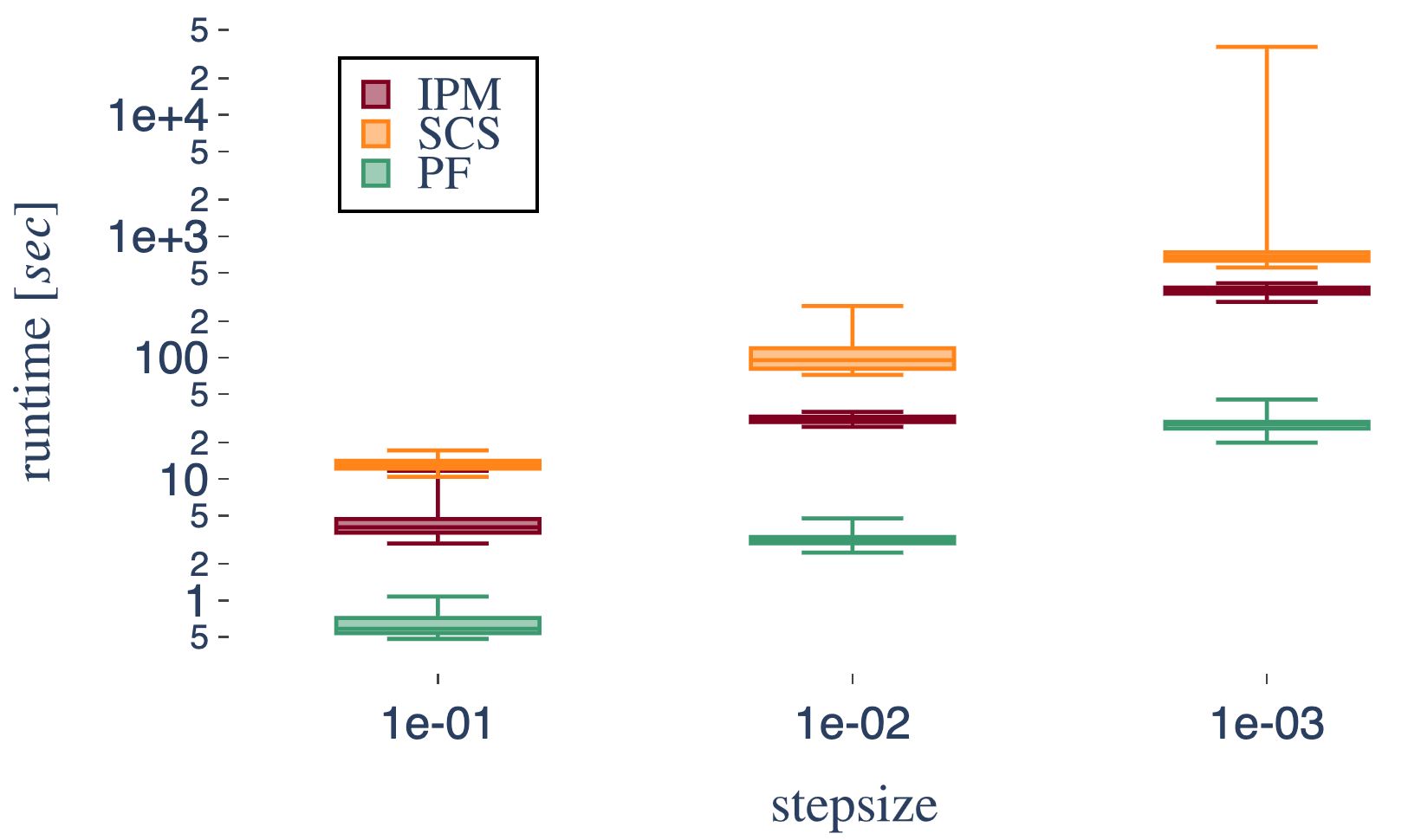}
    \caption{Distribution of the runtime as function of the step size.}
    \label{fig:runtimes}
\end{figure}

\begin{figure}[t] 
\begin{subfigure}{0.49\textwidth}
\centering
\includegraphics[width = \textwidth]{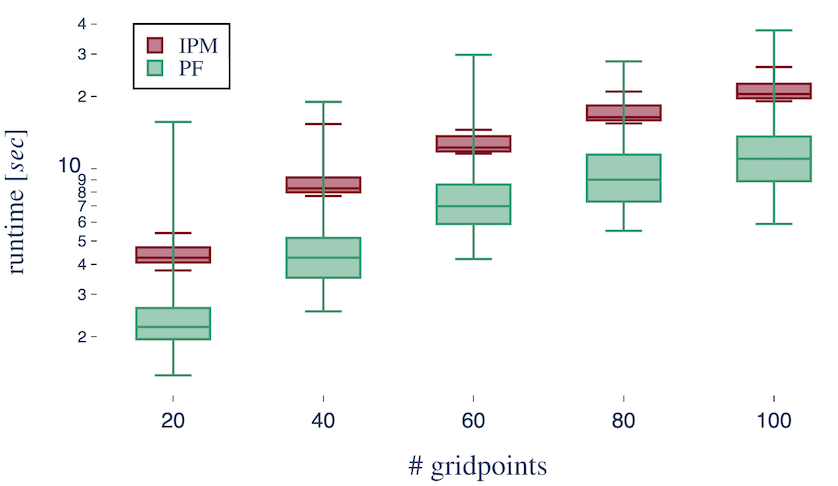}
\caption{Relative gap termination tolerance $= 10^{-9}$}
\label{fig:left}
\end{subfigure}
\begin{subfigure}{0.49\textwidth} 
\includegraphics[width = \textwidth]{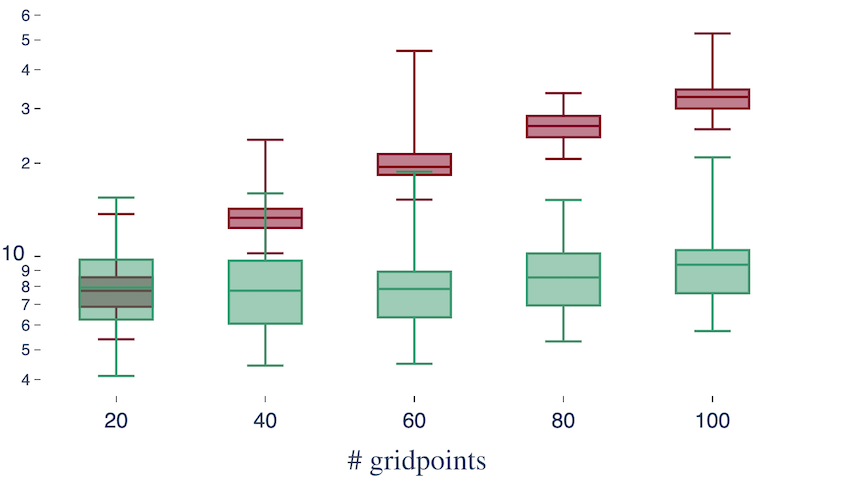}
\caption{Relative gap termination tolerance $= 10^{-15}$}
\label{fig:right}
\end{subfigure}
\caption{Average runtime of MOSEK IPM and Algorithm~\ref{alg: 1} for tracking the TV-SDP solutions with the same residual accuracy on a grid, as a function of the number of gridpoints.}
\label{fig:grid}
\vspace*{-0.2cm}
\end{figure}

Finally, we apply Algorithm~\ref{alg: 1} to the same set of TV-MCR problems allowing for a step size adjustment (setting \texttt{step size\_TUNING} to \texttt{TRUE}).
In order to provide a fair comparison with MOSEK IPM, we fixed five subdivisions of the interval $[0,1]$ in a grid of, respectively, 20, 40, 60, 80, and 100 equidistant points. 
For each grid, at each time point, we used MOSEK with a relative gap termination tolerance of $10^{-14}$ to obtain the corresponding TV-SDP solution, recording the runtime and the average residual over the tracking of each instance. 
For each grid, we then run our algorithm with step size adjustment in order to ensure the same average residual accuracy guaranteed by MOSEK, additionally enforcing the path-following procedure to hit the grid points. 
In this way, we ensure that our procedure has the same accuracy of MOSEK both in terms of the solution residual and of the tracking resolution.

Figure~\ref{fig:grid} shows the distributions of the runtimes as a function of the number of grid points of both Algorithm~\ref{alg: 1} (green) and IPM with two different relative gap termination tolerances: $10^{-9}$ (Figure \ref{fig:grid}(a)) and $10^{-15}$ (Figure \ref{fig:grid}(b)). 

Encouragingly, we observe that we can ensure both the same accuracy and tracking resolution of MOSEK at a smaller average runtime. The constant behavior of the green plot on the right is due to the fact that, in order to ensure the same residual accuracy of the IPM, the path-following procedure needs to consider a number of points that are quite denser then the number of grid points, and hence independent from this latter, while for the plot on the left it is instead sufficient for Algorithm \ref{alg: 1} to follow the grid.

\section{Conclusion}
\label{sec: conclusion}

In this paper, we proposed an algorithm for solving time-varying SDPs based on a path-following predictor-corrector scheme for the Burer--Monteiro factorization.
The restriction to a horizontal space ensures that the linearized KKT conditions system is uniquely solvable under standard regularity assumptions on the TV-SDP problem, thus leading to a well-defined path-following procedure with rigorous error bounds on the distance from the optimal trajectory.
Preliminary numerical experiments on a time-varying version of the max-cut SDP relaxation suggest that our algorithm is competitive both in terms of runtime and accuracy when compared to the application of standard IPMs. Future work should explore the applicability and relative merits of our approach in further applications. 

So far we have assumed that the rank $r$ of the true solution curve is known and remains constant. While this is certainly appropriate for a rigorous analysis as conducted in this work, it might be restrictive in practice. An important extension hence would be to develop rank-adaptive versions of our path-following approach that are able to detect and adjust the appropriate rank in a Burer--Monteiro factorization, for example, by monitoring the smallest singular values of the~matrices $Y_t$.
 
Another important aspect is the initialization of the method, which requires an accurate SDP solution and is currently not based on Burer--Monteiro factorization, thus undermining the computational efficiency of the whole approach. The obvious way out is to also solve the initial time problem using the factorized approach~\cite{Burer2003}. The metaalgorithm presented in~\cite{Journee2010} even does this in a rank-adaptive way. Although this is a nonconvex problem, several works, including also~\cite{Boumal2015,Rosen2021,cifuentes2019polynomial}, have considered Burer--Monteiro schemes with guaranteed and certifiable convergence to a globally optimal low-rank factor under mild conditions, making this a reliable approach in practice.

\section*{Acknowledgments}
The research leading to these results received funding from the OP RDE under Grant Agreement CZ.02.1.01/0.0/0.0/16\_019/0000765. The first author gratefully acknowledges the support of the Czech Science Foundation (grant 22-15524S). The authors also thank two anonymous referees for their helpful comments.
{
\small
\bibliography{references}

\begin{thebibliography}{10}

\bibitem{Aaronson2019}
S.~Aaronson, X.~Chen, E.~Hazan, S.~Kale, and A.~Nayak.
\newblock Online learning of quantum states.
\newblock {\em J. Stat. Mech. Theory Exp., 2019}, pages 124019, 14, 2019.

\bibitem{Absil2008}
P.-A. Absil, R.~Mahony, and R.~Sepulchre.
\newblock {\em Optimization algorithms on matrix manifolds}.
\newblock Princeton University Press, Princeton, NJ, 2008.

\bibitem{Ahmadi2021}
A.~A. Ahmadi and B.~El~Khadir.
\newblock Time-varying semidefinite programs.
\newblock {\em Math. Oper. Res.}, 46(3):1054--1080, 2021.

\bibitem{Alizadeh1997}
F.~Alizadeh, J.-P.~A. Haeberly, and M.~L. Overton.
\newblock Complementarity and nondegeneracy in semidefinite programming.
\newblock {\em Math. Program.}, 77(2, Ser. B):111--128, 1997.

\bibitem{Allgower1990}
E.~L. Allgower and K.~Georg.
\newblock {\em Introduction to numerical continuation methods}.
\newblock SIAM, Philadelphia, 2003.

\bibitem{Andersen2003}
E.~D. Andersen, C.~Roos, and T.~Terlaky.
\newblock On implementing a primal-dual interior-point method for conic
  quadratic optimization.
\newblock {\em Math. Program.}, 95:249--277, 2003.

\bibitem{AndersonPhD}
E.~J. Anderson.
\newblock {\em A Continuous Model For Job-Shop Scheduling}.
\newblock PhD thesis, University of Cambridge, Cambridge, 1978.

\bibitem{Balan2021}
R.~Balan and C.~B. Dock.
\newblock Lipschitz analysis of generalized phase retrievable matrix frames.
\newblock {\em SIAM J. Matrix Anal. Appl.}, 43(3):1518--1571, 2022.

\bibitem{Barvinok1995}
A.~Barvinok.
\newblock Problems of distance geometry and convex properties of quadratic
  maps.
\newblock {\em Discrete Comput. Geom.}, 13(2):189--202, 1995.

\bibitem{Barvinok2001}
A.~Barvinok.
\newblock A remark on the rank of positive semidefinite matrices subject to
  affine constraints.
\newblock {\em Discrete Comput. Geom.}, 25(1):23--31, 2001.

\bibitem{Bellman1953}
R.~Bellman.
\newblock Bottleneck problems and dynamic programming.
\newblock {\em Proc. Nat. Acad. Sci. USA}, 39:947--951, 1953.

\bibitem{Bellon2021}
A.~Bellon, D.~Henrion, V.~Kungurtsev, and J.~Mare\v{c}ek.
\newblock Time-varying semidefinite programming: Geometry of the trajectory of
  solutions.
\newblock {\em arXiv:2104.05445}, 2021.

\bibitem{Boumal2015}
N.~Boumal.
\newblock A {Riemannian} low-rank method for optimization over semidefinite
  matrices with block-diagonal constraints.
\newblock {\em arXiv:1506.00575}, 2015.

\bibitem{Boumal2016}
N.~Boumal, V.~Voroninski, and A.~Bandeira.
\newblock The non-convex {Burer-Monteiro} approach works on smooth semidefinite
  programs.
\newblock In {D. Lee et al.}, editor, {\em Advances in Neural Information
  Processing Systems}, volume~29, pages 2757--2765. Curran Associates, Inc.,
  2016.

\bibitem{Boumal2020}
N.~Boumal, V.~Voroninski, and A.~S. Bandeira.
\newblock Deterministic guarantees for {B}urer-{M}onteiro factorizations of
  smooth semidefinite programs.
\newblock {\em Comm. Pure Appl. Math.}, 73(3):581--608, 2020.

\bibitem{Burer2003}
S.~Burer and R.~D.~C. Monteiro.
\newblock A nonlinear programming algorithm for solving semidefinite programs
  via low-rank factorization.
\newblock {\em Math. Program.}, 95(2, Ser. B):329--357, 2003.

\bibitem{Burer2005}
S.~Burer and R.~D.~C. Monteiro.
\newblock Local minima and convergence in low-rank semidefinite programming.
\newblock {\em Math. Program.}, 103(3, Ser. A):427--444, 2005.

\bibitem{Cifuentes2021}
D.~Cifuentes.
\newblock On the {B}urer-{M}onteiro method for general semidefinite programs.
\newblock {\em Optim. Lett.}, 15(6):2299--2309, 2021.

\bibitem{cifuentes2019polynomial}
D.~Cifuentes and A.~Moitra.
\newblock Polynomial time guarantees for the {Burer-Monteiro} method.
\newblock In {S. Koyejo et al.}, editor, {\em Advances in Neural Information
  Processing Systems}, volume~35, pages 23923--23935. Curran Associates, Red
  Hook, NY, 2022.

\bibitem{Colombo2011}
M.~Colombo, J.~Gondzio, and A.~Grothey.
\newblock A warm-start approach for large-scale stochastic linear programs.
\newblock {\em Math. Program.}, 127(2, Ser. A):371--397, 2011.

\bibitem{Dantzig1977}
G.~B. Dantzig.
\newblock Large-scale systems optimizations with application to energy.
  {Technical report {SOL 77-3}}, 4 1977.

\bibitem{Diehl2012}
Q.~T. Dinh, C.~Savorgnan, and M.~Diehl.
\newblock Adjoint-based predictor-corrector sequential convex programming for
  parametric nonlinear optimization.
\newblock {\em SIAM J. Optim.}, 22(4):1258--1284, 2012.

\bibitem{Engau2010}
A.~Engau, M.~F. Anjos, and A.~Vannelli.
\newblock On interior-point warmstarts for linear and combinatorial
  optimization.
\newblock {\em SIAM J. Optim.}, 20(4):1828--1861, 2010.

\bibitem{Freund2006}
R.~M. Freund.
\newblock On the behavior of the homogeneous self-dual model for conic convex
  optimization.
\newblock {\em Math. Program.}, 106(3, Ser. A):527--545, 2006.

\bibitem{Goemans1995}
M.~X. Goemans and D.~P. Williamson.
\newblock Improved approximation algorithms for maximum cut and satisfiability
  problems using semidefinite programming.
\newblock {\em J. ACM}, 42(6):1115--1145, 1995.

\bibitem{Goldfarb1999}
D.~Goldfarb and K.~Scheinberg.
\newblock On parametric semidefinite programming.
\newblock In {\em Proceedings of the {S}tieltjes {W}orkshop on {H}igh
  {P}erformance {O}ptimization {T}echniques ({HPOPT} '96)}, pages 361--377.
  Elsevier, Amsterdam, Appl. Numer. Math. 29, 1999.

\bibitem{gondzio2002reoptimization}
J.~Gondzio and A.~Grothey.
\newblock Reoptimization with the primal-dual interior point method.
\newblock {\em SIAM J. Optim.}, 13(3):842--864, 2002.

\bibitem{Gondzio2008}
J.~Gondzio and A.~Grothey.
\newblock A new unblocking technique to warmstart interior point methods based
  on sensitivity analysis.
\newblock {\em SIAM J. Optim.}, 19(3):1184--1210, 2008.

\bibitem{Guddat1990}
J.~Guddat, F.~Guerra~Vazquez, and H.~T. Jongen.
\newblock {\em Parametric optimization: singularities, pathfollowing and
  jumps}.
\newblock B. G. Teubner, Stuttgart; John Wiley \& Sons, Ltd., Chichester, 1990.

\bibitem{Hauenstein2022}
J.~D. Hauenstein, A.~Mohammad-Nezhad, T.~Tang, and T.~Terlaky.
\newblock On computing the nonlinearity interval in parametric semidefinite
  optimization.
\newblock {\em Math. Oper. Res.}, 47(4):2989--3009, 2022.

\bibitem{henzinger2022practical}
M.~Henzinger, A.~Noe, and C.~Schulz.
\newblock Practical fully dynamic minimum cut algorithms.
\newblock In {\em 2022 Proceedings of the Symposium on Algorithm Engineering
  and Experiments (ALENEX), SIAM, Phildelphia}, pages 13--26, 2022.

\bibitem{Horn2013}
R.~A. Horn and C.~R. Johnson.
\newblock {\em Matrix analysis}.
\newblock Cambridge University Press, Cambridge, second edition, 2013.

\bibitem{Im2021}
J.~Im and H.~Wolkowicz.
\newblock A strengthened {B}arvinok-{P}ataki bound on {SDP} rank.
\newblock {\em Oper. Res. Lett.}, 49(6):837--841, 2021.

\bibitem{jarre1993interior}
F.~Jarre.
\newblock An interior-point method for minimizing the maximum eigenvalue of a
  linear combination of matrices.
\newblock {\em SIAM J. Control Optim.}, 31(5):1360--1377, 1993.

\bibitem{Jiang2020}
H.~Jiang, T.~Kathuria, Y.~T. Lee, S.~Padmanabhan, and Z.~Song.
\newblock A faster interior point method for semidefinite programming.
\newblock In {\em 2020 {IEEE} 61st {A}nnual {S}ymposium on {F}oundations of
  {C}omputer {S}cience}, pages 910--918. IEEE Computer Society, Los Alamitos,
  CA, 2020.

\bibitem{Journee2010}
M.~Journ\'{e}e, F.~Bach, P.-A. Absil, and R.~Sepulchre.
\newblock Low-rank optimization on the cone of positive semidefinite matrices.
\newblock {\em SIAM J. Optim.}, 20(5):2327--2351, 2010.

\bibitem{Kao2017}
E.~Kao, V.~Gadepally, M.~Hurley, M.~Jones, J.~Kepner, S.~Mohindra,
  P.~Monticciolo, A.~Reuther, S.~Samsi, W.~Song, D.~Staheli, and S.~Smith.
\newblock Streaming graph challenge: Stochastic block partition.
\newblock In {\em 2017 IEEE High Performance Extreme Computing Conference
  (HPEC), IEEE, Piscataway, NJ}, pages 1--12, 2017.

\bibitem{Lavaei2011}
J.~Lavaei and S.~H. Low.
\newblock Zero duality gap in optimal power flow problem.
\newblock {\em IEEE Transactions on Power Systems}, 27(1):92--107, 2012.

\bibitem{Massart2020}
E.~Massart and P.-A. Absil.
\newblock Quotient geometry with simple geodesics for the manifold of
  fixed-rank positive-semidefinite matrices.
\newblock {\em SIAM J. Matrix Anal. Appl.}, 41(1):171--198, 2020.

\bibitem{MOSEK}
{MOSEK ApS}.
\newblock {\em {The MOSEK optimization toolbox for MATLAB manual. Version
  9.3.}}, 2019.

\bibitem{Nazarathy2009}
Y.~Nazarathy and G.~Weiss.
\newblock Near optimal control of queueing networks over a finite time horizon.
\newblock {\em Ann. Oper. Res.}, 170:233--249, 2009.

\bibitem{Nesterov2018}
Y.~Nesterov.
\newblock {\em Lectures on convex optimization}.
\newblock Springer, Cham, Switzerland, 2018.

\bibitem{Nocedal2006}
J.~Nocedal and S.~Wright.
\newblock {\em Numerical optimization}.
\newblock Springer, New York, second edition, 2006.

\bibitem{Donoghue2021}
B.~O'Donoghue.
\newblock Operator splitting for a homogeneous embedding of the linear
  complementarity problem.
\newblock {\em SIAM J. Optim.}, 31(3):1999--2023, 2021.

\bibitem{Donoghue2016}
B.~O'Donoghue, E.~Chu, N.~Parikh, and S.~Boyd.
\newblock Conic optimization via operator splitting and homogeneous self-dual
  embedding.
\newblock {\em J. Optim. Theory Appl.}, 169(3):1042--1068, 2016.

\bibitem{scs}
B.~O'Donoghue, E.~Chu, N.~Parikh, and S.~Boyd.
\newblock {SCS}: Splitting {C}onic {S}olver, version 3.2.2.
\newblock \url{https://github.com/cvxgrp/scs}, Nov. 2022.

\bibitem{Pataki1998}
G.~Pataki.
\newblock On the rank of extreme matrices in semidefinite programs and the
  multiplicity of optimal eigenvalues.
\newblock {\em Math. Oper. Res.}, 23(2):339--358, 1998.

\bibitem{Rosen2021}
D.~M. Rosen.
\newblock Scalable low-rank semidefinite programming for certifiably correct
  machine perception.
\newblock In {\em Algorithmic Foundations of Robotics {XIV}}, pages 551--566.
  Springer, Cham, Switzerland, 2021.

\bibitem{Rosen2020}
D.~M. Rosen, L.~Carlone, A.~S. Bandeira, and J.~J. Leonard.
\newblock {\em A certifiably correct algorithm for synchronization over the
  special Euclidean group}, pages 64--79.
\newblock in Algorithmic Foundations of Robotics XII, Springer, Cham,
  Switzerland, 2020.

\bibitem{Schmitt1992}
B.~A. Schmitt.
\newblock Perturbation bounds for matrix square roots and {P}ythagorean sums.
\newblock {\em Linear Algebra Appl.}, 174:215--227, 1992.

\bibitem{Skajaa2013}
A.~Skajaa, E.~D. Andersen, and Y.~Ye.
\newblock Warmstarting the homogeneous and self-dual interior point method for
  linear and conic quadratic problems.
\newblock {\em Math. Program. Comput.}, 5(1):1--25, 2013.

\bibitem{Teren1977}
F.~Teren.
\newblock Minimum time acceleration of aircraft turbofan engines by using an
  algorithm based on nonlinear programming.
\newblock In {\em NASA Technical Memorandum TM-73741, Lewis Research Center,
  Cleveland, Ohio}, September, 1977.

\bibitem{Tuncel2000}
L.~Tun\c{c}el.
\newblock Potential reduction and primal-dual methods.
\newblock In {\em Handbook of semidefinite programming}, volume~27 of {\em
  Internat. Ser. Oper. Res. Management Sci.}, pages 235--265. Kluwer Acad.,
  Boston, MA, 2000.

\bibitem{Waldspurger2020}
I.~Waldspurger and A.~Waters.
\newblock Rank optimality for the {B}urer-{M}onteiro factorization.
\newblock {\em SIAM J. Optim.}, 30(3):2577--2602, 2020.

\bibitem{Wang2009}
X.~Wang, S.~Zhang, and D.~D. Yao.
\newblock Separated continuous conic programming: strong duality and an
  approximation algorithm.
\newblock {\em SIAM J. Control Optim.}, 48(4):2118--2138, 2009.

\bibitem{handbookSDP}
H.~Wolkowicz, R.~Saigal, and L.~Vandenberghe, editors.
\newblock {\em Handbook of semidefinite programming}.
\newblock Kluwer Academic, Boston, MA, 2000.

\bibitem{Wright2005}
S.~J. Wright.
\newblock An algorithm for degenerate nonlinear programming with rapid local
  convergence.
\newblock {\em SIAM J. Optim.}, 15(3):673--696, 2005.

\end{thebibliography}
}
\end{document}